\documentclass[10pt]{article}

\usepackage{amsfonts}
\usepackage{amsthm}
\usepackage{amsmath}
\usepackage{cite}
\usepackage{amscd}
\usepackage{amssymb,graphicx}
\usepackage[utf8]{inputenc}
\usepackage[mathscr]{eucal}
\usepackage{indentfirst}
\usepackage{graphics}
\usepackage{caption}
\usepackage{subcaption}
\usepackage{pict2e}
\usepackage{epic}
\usepackage{braket}
\usepackage{color}
\usepackage{mathtools}
\RequirePackage[numbers]{natbib}
\RequirePackage[colorlinks,citecolor=blue,urlcolor=blue]{hyperref}

\newtheorem{theorem}{Theorem}
\newtheorem{lemma}{Lemma}

\newtheorem{proposition}{Proposition}
\theoremstyle{definition}
\newtheorem{definition}{Definition}

\newtheorem{remark}{Remark}

\numberwithin{equation}{section}

\newtheorem{manuallemmainner}{Lemma}
\newenvironment{manuallemma}[1]{%
  \renewcommand\themanuallemmainner{#1}%
  \manuallemmainner
}{\endmanuallemmainner}

\newtheorem{manualtheoreminner}{Theorem}
\newenvironment{manualtheorem}[1]{%
  \renewcommand\themanualtheoreminner{#1}%
  \manualtheoreminner
}{\endmanualtheoreminner}

\begin{document}


\title{On almost sure limit theorems for heavy-tailed products of long-range dependent linear processes\footnote{This research has been conducted at the University of Alberta. Michael A. Kouritzin has been supported by a NSERC Discovery Grant, and Sounak Paul by a UAlberta Graduate Recruitment Scholarship, Pundit RD Sharma Memorial Scholarship, and Josephine M. Mitchell Scholarship.}}



	
	

\author{Michael A. Kouritzin \\
    Department of Mathematical and Statistical Sciences\\
	University of Alberta\\
	\and 
	Sounak Paul \\
	Department of Statistics\\
	University of Chicago\\
	}

\date{\vspace{-5ex}}

\maketitle


\begin{abstract}
Marcinkiewicz strong law of large numbers, ${n^{-\frac1p}}\sum_{k=1}^{n} (d_{k}- d)\rightarrow 0\ $ almost surely with $p\in(1,2)$, are developed for products $d_k=\prod_{r=1}^s x_k^{(r)}$, 
where the $x_k^{(r)} = \sum_{l=-\infty}^{\infty}c_{k-l}^{(r)}\xi_l^{(r)}$ are two-sided 
linear process with coefficients $\{c_l^{(r)}\}_{l\in \mathbb{Z}}$ and i.i.d. 
zero-mean innovations $\{\xi_l^{(r)}\}_{l\in \mathbb{Z}}$. 
The decay of the coefficients $c_l^{(r)}$ as $|l|\to\infty$, can be slow enough for $\{x_k^{(r)}\}$ to have long memory while $\{d_k\}$ can have heavy tails. 
The long-range dependence and heavy tails for $\{d_k\}$ are handled simultaneously and a decoupling property shows the convergence rate is dictated by the worst of long-range dependence and heavy tails, but not their combination.
The Marcinkiewicz strong law of large numbers is also extended to the multivariate linear process case.
\end{abstract}


\tableofcontents

\section{Introduction}\label{chap1}
\subsection{Background}\label{bkgrnd}
\noindent
With today’s internet of things, big data has become abundant and huge opportunities await those who can effectively mine it. 
However, this data, especially in finance, econometrics, networks, machine learning, signal processing, and environmental science, often posseses heavy-tails and long memory (see \cite{relevance,runoff,patagonia,stockcrash}). 
Data exhibiting this combination of heavy-tails (HT) and long-range dependence (LRD) can often be modeled by linear processes but is lethal for most classical statistics. 
Recently, certain covariance estimators and stochastic approximation algorithms have been shown capable of handling this kind of data. 
In particular, Marcinkiewicz strong laws of large numbers (MSLLN) were established for showing polynomial rates of convergence (see \cite{KOSA,louhichi,Rio}).
The point of this paper is that, if one establishes MSLLNs for finite products of a data
stream, then the implied polynomial rates can be used to quantify the amount (if any) of
LRD and HT the data stream exhibits.

The tails of HT distributions are not exponentially bounded and estimating the tail decay is a common problem. Useful subclasses of HT distributions include subexponential distributions (which possess a stronger regularity condition on their tails, and were studied in \cite{chistyakov, Teugels}), and Lévy $\alpha$-stable distributions (with $\alpha < 2$), whose significance lie in generalizing the central limit theorem. For HT random variables the normalized cumulative-sum distributional limit is often a non-normal stable distribution, referred to by Mandelbrot \cite{Mand60,Mand61} as \textit{stable Paretian distribution}. Several stable distributions, such as Pareto, Lévy, and Weibull, are used in financial models. Heavy-tailed stochastic processes and their extreme value theory, have historically been a vibrant field of study (see Kulik and Soulier \cite{kuliksoulier}). In comparison to HT, LRD is a phenomenon that came to prominence more recently. Indication of long memory in environmental and hydrological time series drew a lot of attention in the mid-twentieth century, especially in fluid flow models  (see \cite{Hurst,otherhurst,noahjoseph}). Today, the LRD-HT combination frequently appears in fluid flow (see \cite{runoff,patagonia}), network traffic (see \cite{relevance,internet}), finance and stock markets (see \cite{stockcrash,mingliu}), particularly in \emph{stock market volatility} financial models.


A detailed history of LRD and HT can be found in \cite{history}. Hosking \cite{Hosking} laid the foundation for the class of ARFIMA (Autoregressive fractionally integrated moving average) models, which are now often used to simulate this combination. HT along with LRD also influence the amount of self-similarity (see Pipiras and Taqqu \cite{taqqubook}), a property which forms the basis for fractals, observed in time series. Autocovariance estimation under LRD and HT is also a field of great importance, owing to the widespread use of autocovariance functions (see \cite{Davisresnick,crosscov,Wudepinnov}). Limit theorems for sample covariances of linear processes with i.i.d innovations having regularly varying tail probabilities, was studied in \cite{Davisresnick}. Kouritzin \cite{crosscov} studied strong Gaussian approximations for cross-covariances of causal linear processes with finite fourth moments, and independent innovations. Wu et al. \cite{covlongmem}, Wu and Min \cite{Wudepinnov} studied the asymptotic behavior of sample covariances of linear processes with weakly dependent innovations, and provided both central and non-central limit theorem for the same.

Very few MSLLN results have been explored for the combination of LRD and HT data. Louhichi and Soulier \cite{louhichi} gave a MSLLN for linear processes where the innovations  are linear symmetric $\alpha$-stable processes, and with coefficients $\{c_i\}_{i\in\mathbb{Z}}$ satisfying $\sum_{i=-\infty}^{\infty}|c_i|^s<\infty$ for some $1\leq s<\alpha$. Rio \cite{Rio} explored MSLLN results for a strongly mixing sequence $\{X_n\}_{n\in\mathbb{Z}}$ assuming conditions on the mixing rate function and the quantile function of $|X_0|$. Dang and Istas \cite{dangistas} obtained consistent estimators for both the Hurst as well as stability indices of $H$-self-similar $\alpha$-stable processes. Kouritzin and Sadeghi \cite{KOSA} gave a MSLLN for the outer product of two-sided linear processes exhibiting both long memory and heavy tails, and found that the rate of convergence differed from that of linear process alone. This led us to believe that MSLLN for higher products might have different rates, and quantifying their rates of convergence could lead to interesting applications like devising simple tests to indicate presence of LRD and HT in data. Indeed, by applying Proposition \ref{equalinnov} of our paper, with different powers, and observing where convergence and divergence takes place, one could get an indication of the range of LRD and HT present in the dataset. This is a potential area for further investigation. Generalizing \cite[Theorem~3]{KOSA} from outer to arbitrary products will be the main goal of this paper. More motivation and explanation of challenges faced, is provided in Section \ref{motiv}. We refer the reader to \cite{kosadecoup} for possible further applications of our results to stochastic approximation and observer design.

\subsection{Notation and Definitions}\label{list}
\noindent
The following notation and conventions will be used throughout the paper.\\
$\bullet\ \ \|A\|_F$ is the Frobenius norm of $A$, i.e. $\sqrt{trace(A^TA)}$ for any matrix $A\in\mathbb R^{m\times n}$, where $m,n \in \mathbb{N}$.\\
$\bullet\ \ \|X\|_p=\left[E\left(X^p\right)\right]^\frac{1}{p}$ for any non-negative random variable $X$, and $p>0$.\\
$\bullet\ $ For vectors $v^{(r)}\in\mathbb R^d,\ 1\leq r\leq n$, $d\in\mathbb{N}$, we define their tensor product $\substack{n\\\bigotimes\\r=1}v^{(r)}\ \in\ \mathbb{R}^{d^n}$ element-wise, as
$$\left(\substack{n\\\bigotimes\\r=1}v^{(r)}\right)_{i_1 i_2\hdots i_n}\ =\ \prod_{r=1}^{n}v^{(r)}_{i_r},\qquad 1\leq i_j\leq d,\ \forall\ 1\leq j\leq n.$$
$\bullet\ \ a_{i,k}\stackrel{i}\ll b_{i,k}$ means that for each $k,\ \ \exists\ c_k>0$ that does not depend upon $i$ such that $|a_{i,k}|\le c_k |b_{i,k}|$ for all $i,k$ (also used in \cite{interrelation,KOSA}).\\
$\bullet\ \ l_{n,\beta}\left(x\right)\ \ 
=\ \ \left\{ \begin{array}{ll}
x^{n(1-2\beta)+1},&\beta <\frac{n+1}{2n}\\
\log(x+1),&\beta =\frac{n+1}{2n}\\
1,&\beta >\frac{n+1}{2n}
\end{array}\right.$,$\quad\ \forall\ n\in \mathbb{N} \mbox{ and } \beta \in \mathbb{R}$.\\
$\bullet\ \ l_i$ shall denote the $i$th coordinate of the vector $\boldsymbol\ell\in\mathbb{Z}^q$, for $q\in\mathbb{N},\ 1\leq i\leq q$. In other words, $\boldsymbol\ell\ =\ \left(l_1,l_2,\hdots,l_q\right)$.\\
$\bullet\ \ \mathcal P_s$ denotes the collection of permutations of $\{1,2,\hdots,s\}$.\\ 
$\bullet\ $ If $\{f_r\}_{r\in\mathbb{Z}}$ is a sequence of functions or constants, and $a,b\in\mathbb{N}\cup\{0\}$ such that $a>b$, then $\prod_{r=a}^{b}f_r=1$.\\
$\bullet\ $ If $x\geq 0$ and $a>0$, then at the point $x = 0$, $a\wedge\frac{1}{x}=\lim_{x\rightarrow 0^+} a\wedge \frac{1}{x}$.\\

Our standard notation includes: $|x|$ is Euclidean norm of $x\in \mathbb R^d$, $\ \textbf{1}_A$ is the indicator function of the event A, $|S|$ is the cardinality of the set $S$, $a\vee b=\max\{a,b\}$, $a\vee b\vee c=\max\{a,b,c\}$, $\ a\wedge b=\min\{a,b\}$, $\ a\wedge b\wedge c=\min\{a,b,c\}$, $\ a\vee b\wedge c=(a\vee b)\wedge c$, $\lfloor c\rfloor$ and $\lceil c\rceil$ are the greatest and least integer functions of $c\in\mathbb{R}$ respectively.

Now, we formally define the basic concepts that will be used throughout the paper. 
The Marcinkiewicz strong law of large numbers is defined in Appendix \ref{appclassic}. 
We use the following weak HT definition, also used in \cite{KOSA}, that basically says that the tails decay like $x^{-\beta}$ for some real number $\beta$.
\begin{definition}[Heavy tails]\label{htdef}
	A random variable $X$ is said to be heavy-tailed, if
	\begin{eqnarray}\nonumber
	\beta\ =\ \sup\ \left\{q\geq 0:\ \sup_{x\geq 0}\ x^q P\left(|X| > x \right) < \infty\right\}\ \ <\ \infty\ ,
	\end{eqnarray}
	and $\beta$ will be called the heavy-tail coefficient of $X$.
\end{definition}
Notice $\beta > p$ implies that $E[|X|^p]<\infty$ and the classical MSLLN in Theorem \ref{msllndef} of Appendix \ref{appclassic} holds. 
The smaller the value of $\beta$, the heavier the tail of $X$.  

Five non-equivalent LRD conditions are provided and compared in \cite[Chapter~2]{taqqubook}, that could be used as a definition of LRD. Since we only treat time series with linear representations, their first condition is most natural to us. Still, we shall use a more general, two-sided version of their first condition as our definition of LRD. We first provide the definition of slowly varying sequence.
\begin{definition}[Slowly varying sequence]\label{slowvary}
    A sequence $\{L(n)\}_{n\in\mathbb{N}}$ is said to be slowly varying if it is positive for $n\geq n_0$ for some $n_0\in\mathbb{N}$, and
    $$\lim_{n\rightarrow\infty}\frac{L(\lfloor an\rfloor)}{L(n)}=1,\quad \forall\ a>0\ .$$
\end{definition}
\begin{definition}[Long-range dependence]\label{lrddef}
	The time series $X=\{X_n\}_{n\in\mathbb{Z}}$, with linear representation
	$$X_n = \mu\ +\ \sum_{l=-\infty}^{\infty}c_{n-l}\xi_{l},$$
	where $\mu\in\mathbb{R}$, and $\{\xi_l\}_{l\in\mathbb{Z}}$ are uncorrelated random variables with zero mean and common variance, is \textit{long-range dependent} if $\{c_l\}_{l\in\mathbb{Z}}$ are real coefficients satisfying
	\begin{eqnarray}\nonumber
	|l|^\sigma c_l = \begin{cases}
      L_1(l) & \text{if}\ l\in \{1,2,3,\hdots\},\\
      L_2(-l) & \text{if}\ l\in \{-1,-2,-3,\hdots\},
    \end{cases}
	\end{eqnarray}
	for some $\sigma\in\left(\frac12,1\right)$, and some slowly varying sequences $L_1$ and $L_2$.
	A smaller $\sigma$ indicates longer range dependence and $\sigma\ge 1$ indicates no long-range dependence. 
\end{definition}
According to \cite{taqqubook}, Definition \ref{lrddef} implies that the autocovariance function of the LRD time series $X$, i.e. $\gamma_X(k)=E[X_0X_k]\ $, will be equal to $k^{1-2\sigma}\overline{L}(k)$, where $\overline{L}$ is another slowly varying sequence, and that these autocovariances are not absolutely summable.

\noindent
\textbf{Note:} Herein, since we are only considering linear processes, we further assume that the innovations $\{\xi_l\}_{l\in\mathbb{Z}}$ are i.i.d. random variables.

\section{Motivation and Results}\label{motiv}
\noindent
In this section, we introduce arbitrary products and powers of $\mathbb{R}$-valued linear processes, for which we will establish MSLLN.
We also motivate the conditions required to establish these results.
Finally, at the end of the section we give a multivariate generalization.

\noindent
{\bf General $\mathbb{R}$-valued product case}:
Let $s\in\mathbb N$ and $\ \left\{\left(x_k^{(1)},x_k^{(2)},\hdots,x_k^{(s)}\right)\right\}_ {k\in\mathbb{Z}}\ $ be $\mathbb{R}^{s}$-valued random vectors, with
\begin{eqnarray}\label{linprocess}
x_k^{(r)} = \sum_{l=-\infty}^{\infty} c_{k-l}^{(r)}\xi_{l}^{(r)},\qquad\ \ \forall\ 1\leq r\leq s,
\end{eqnarray}
being two-sided linear processes in terms of $\mathbb{R}^{s}$-valued i.i.d. innovation vectors $\left\{\left(\xi_l^{(1)},\xi_l^{(2)},\hdots,\xi_l^{(s)}\right)\right\}_{l\in \mathbb{Z}}\ $ with zero-mean and finite variance, and coefficients $\left\{\left(c_l^{(1)},c_l^{(2)},\hdots,c_l^{(s)}\right)\right\}_{l\in \mathbb{Z}}$ satisfying some decay condition (see below). The finite variance assumption, along with the conditions (Reg, Tail, Decay) that we introduce later,  ensures the almost sure convergence of (\ref{linprocess}). Notice that we are not assuming any dependence structure among the variables $\xi_l^{(1)},\xi_l^{(2)},\hdots,\xi_l^{(s)}$ for any fixed $l$. The coefficients $\{c_l^{(i)}\}_{l\in\mathbb{Z}}$ may decay slowly enough that $\{x_k^{(i)}\}_{l\in\mathbb{Z}}$ has LRD, for any (or all) $i\in\{1,...,s\}$. Define, 
\begin{equation}\label{dkdef}
d_k= \prod_{r=1}^{s} x_k^{(r)},\ \quad \text{i.e. } d_k=(x_k)^s\ \text{ when }x_k^{(r)}=x_k,\
\forall\ r,
\end{equation} 
and observe that $d_k$ can possess heavy tails in this setting.

\noindent
{\bf $\mathbb{R}$-valued power case}:
This is a special case of the general $\mathbb{R}$-valued product, which is easier to follow. In this case, we still have $s\in\mathbb N$, but $\xi_{l}^{(r)}=\xi_{l}$, $c_l^{(r)}=c_l$ for $r\in\{1,...,s\}$, $l\in\mathbb Z$ so 
\begin{eqnarray}\label{linprocesssimp}
x_k^{(r)} =x_k= \sum_{l=-\infty}^{\infty} c_{k-l}\xi_{l},\qquad\ \ \forall\ 1\leq r\leq s.
\end{eqnarray} 
We impose the following conditions for this case:
\begin{itemize}
\item[(reg]$\!\!\!$)\quad\ 
$E \left[\left|\xi_1\right|^{ 2}\right]<\infty,$
\item[(tail]$\!\!\!$)\quad\ 
$\sup_{t\geq 1}\ t^{\alpha}P\left(|\xi_1|^s > t\right) < \infty\ ,$
for some $\ \alpha > 1$,
\item[(dec]$\!\!\!$ay)
$\sup_{l\in\mathbb Z}\ |l|^{\sigma} \left|c_l\right|<\infty \qquad \mbox{for some } \sigma \in \left(\frac12,1\right].$
\end{itemize}
These conditions allow longer-range dependence for smaller $\sigma$ and 
 heavy-enough tails when $s\ge 2$ and
$\alpha\in(1,2)$ that the second moment of $d_k$ will not exist.
Since there is no slowly varying function in (decay), $\sigma=1$ handles the non-long-range dependence case.

To further motivate Theorem \ref{mostgen} (to follow), we first state the following proposition, which is set in the power case. Notice that $E\left[\left|\xi_1\right|^{s}\right]<\infty$ by (tail) so Condition (Reg) below for Theorem \ref{mostgen} holds.
\begin{proposition}\label{equalinnov}
Assume Conditions (reg), (tail) and (decay) hold, and $x_k$ is defined as in (\ref{linprocesssimp}).
Then, $\ \lim\limits_{n\rightarrow \infty}{n^{-\frac{1}{p}}}\sum\limits_{k=1}^{n}\left((x_k)^s - E\left[(x_k)^s\right]\right)=0\ \ \mbox{a.s,}\ $ for all
\begin{eqnarray}\label{ccond4}
0\ \ <\	\ p\ \ <\ \ \left\{ \begin{array}{ll} \frac{2}{3-2\sigma}, & s=1 \\ 2 \wedge \alpha \wedge \frac{1}{2-2\sigma}, & s=2 \\ \alpha \wedge \frac{2}{3 - 2\sigma}, & s>2 \end{array}\right.\ .
\end{eqnarray}
Furthermore, if $\xi_1$ is a symmetric random variable, and s is even, then the constraint for (\ref{ccond4}) can be relaxed to $0\ <\ p\ <\ 2 \wedge \alpha \wedge \frac{1}{2 - 2\sigma}\ .$
\end{proposition}

\begin{proof}
	The proof of this proposition follows directly from Theorem \ref{mostgen}, with $\xi_{l}^{(r)}=\xi_{l}$, and $c_l^{(r)}=c_l$ for $r\in\{1,...,s\}$.
\end{proof}

\begin{remark}
Due to the power case condition (reg), there cannot be HT influence when $s=1$. Further, if ($\sigma=1$ and $s=1$) or ($s\geq 2$, $\alpha\ge 2$ and $\sigma\ge 1$), then there is neither HT nor LRD and $p$ in (\ref{ccond4}) can be anything less than $2$, which is consistent with classical MSLLN (see Theorem \ref{msllndef}). Note when $s=2$ and $\sigma = 1$, we have $2 \wedge \alpha \wedge \frac{1}{2-2\sigma} = 2 \wedge \alpha$ by the last convention in Subsection \ref{list}.
\end{remark}

\subsection{Main Results}\label{chap2}
\noindent
Our first main result generalizes Proposition \ref{equalinnov}
from powers to products.
For products, the regularity, tail and decay conditions become:
\begin{itemize}
\item[(Reg]$\!\!\!$)
$E \left[\left|\xi_1^{(r)}\right|^{s\vee 2}\right]<\infty\qquad\forall\ 1\leq r\leq s,$
\item[(Tail]$\!\!\!$)
$\max_{\pi\in\mathcal P_s}\ \max_{0 \leq i \leq \left\lfloor \frac{s-1}{2} \right\rfloor}\ \sup_{t\geq 1}\ t^{\alpha_i}P\left(\prod_{r\in\{\pi(1),\hdots,\pi(s-i)\}}\!\!\left|\xi_1^{(r)}\right| > t\right) < \infty\ ,$
for some $\ \alpha_0 > 1,\ \alpha_i =\frac{s}{s-i}\alpha_0\ \mbox{ for }\ i\in\left\{1,2,\hdots,\left\lfloor \frac{s-1}{2} \right\rfloor\right\}$,
\item[(Dec]$\!\!\!$ay)  
$\sup_{l\in\mathbb Z}\ |l|^{\sigma_r} \left|c_l^{(r)}\right|<\infty \qquad \mbox{for some } \sigma_r \in \left(\frac12,1\right],\ \ \forall\ 1\leq r\leq s.$
\end{itemize}

(Reg) ensures existence of the linear process
product and its mean (see the Khinchin-Kolmogorov Theorem in e.g. Shiryaev \cite[Chapter~4, Section~2, Theorem~2]{Shiryaev} or else \cite[Theorem~1.4.1]{Samorod}).

\begin{remark}
$\sigma_r\in \left(\frac{1}{2},1\right)$ allows for the presence of long memory in $x_k^{(r)}$ (see Definition \ref{lrddef}). (Tail) does not necessarily imply the $s$ moment in (Reg) since we do not assume any particular dependence in $r\rightarrow \xi^{(r)}_1$.
For example, if $s=3$, then $\left\lfloor \frac{s-1}{2} \right\rfloor=1$ and we just need $\alpha_1>\frac32$ and (Tail) would imply a moment greater than $\frac32$ on any product $\xi^{(r_1)}_1\xi^{(r_2)}_1$ for $r_1\ne r_2$ but $\xi^{(r_1)}_1$ and $\xi^{(r_2)}_1$ could be independent so this does not imply a third moment on either.  Similarly, $\alpha_0>1$ would only necessarily guarantee more than a first moment.
\end{remark}

\begin{remark}
The products of the linear processes produce sums of products of innovations $\xi_{i_1}^{(1)}\xi_{i_2}^{(2)}\cdots\xi_{i_s}^{(s)}$, where any number of the $i_j$'s may be equal.
$\alpha_i$ in (Tail) is used to control the amount of HT present in terms with $s-i$ innovations having same subscripts. Clearly $\alpha_i$ must get larger with increasing $i$, since the product of fewer innovation at the same time produces lighter HT.
Indeed, in the case where all $\xi^{(r)}_i=\xi_i$ are the same (as for our earlier power Proposition \ref{equalinnov}) (Tail) collapses down to (tail) due to our assignment $\alpha_i=\frac{s}{s-i}\alpha_0$. This assignment is motivated by the case when $\xi_{1}^{(1)}=\hdots=\xi_{1}^{(s)}=\xi_{1}$, where the tail condition $\ \sup_{t\geq 0}\ t^{\alpha_0} P\left(\left|\xi_1\right|^s > t \right) < \infty\ $ implies that $\ \sup_{t\geq 0}\ t^{\frac{s}{s-i}\alpha_0} P\left(\left|\xi_1\right|^{s-i} > t\right) < \infty\ $.
\end{remark}

\begin{theorem}\label{mostgen}
Assume Conditions (Reg), (Tail) and (Decay) hold, $d_k$ is defined as in (\ref{dkdef}), and $d=E[d_1]$.
Then, $\ \lim\limits_{n\rightarrow \infty}{n^{-\frac{1}{p}}}\sum\limits_{k=1}^{n}\left(d_{k}-d\right)=0\ \ \mbox{a.s.}\ \ $for
	\begin{eqnarray}\label{cond4}
	0\ \ <\ \ p\ \ <\ \ \left\{ \begin{array}{ll} \frac{2}{3-2\sigma_1}, & s=1 \\ 2 \wedge \alpha_0 \wedge \frac{1}{2 -\sigma_1 -\sigma_2}, & s=2 \\ \alpha_0 \wedge \frac{2}{3 - 2\min_{1\leq i \leq s}\{\sigma_i\}}, & s>2 \end{array}\right.\ .
	\end{eqnarray}
	Furthermore, if $\ \xi_1^{(1)}=\xi_1^{(2)}=\hdots=\xi_1^{(s)}$, $\xi_1^{(1)}$ is a symmetric random variable, and $s$ is even then the constraint in (\ref{cond4}) can be relaxed to
	\begin{eqnarray}\label{cond5}
	0\ \ <\ \ p\ \ <\ \ 2 \wedge \alpha_0 \wedge \frac{1}{2 - \min_{1\leq i<j\leq s}\{\sigma_i+\sigma_j\}}\ .
	\end{eqnarray}
	
	\vspace{0.3cm}
\end{theorem}
Our linear processes are two sided so both the past and the future must be considered. 
LRD implies absence of strong mixing and HT invalidates direct use of moments techniques.
Thus, we have used a technique to decompose products of sums into subsets based upon how they would contribute to an overall bound.
Definition \ref{sets} below, used in the proofs of Lemmas \ref{covlemma} and \ref{psilemma}, is the basis of this technique. 
This division idea is not completely new but rather related to
earlier decompositions in Bai and Taqqu \cite[Proposition~3.3]{baitaqqu} and Peccati
and Taqqu \cite[Chapter 7]{peccatitaqqu}.

\noindent
{\bf Note on optimality of rates of convergence in Theorem \ref{mostgen}}:
Ideally, Marcinkiewicz strong law of large numbers establish the best polynomial convergence rate. However, proving optimality under heavy-tails and long-range dependence conditions requires establishing central and non-central limit type results. Surgailis \cite{Surgzones,Surglrdappell,Surgemp,Surgstable} established some such results, starting in \cite{Surgzones}, where he studied limit distributions of
\begin{eqnarray}\label{sfloort}
S_{n,h}(t) = \sum_{k=1}^{{\lfloor nt\rfloor}}\left[h(x_k) - E(h(x_k))\right].
\end{eqnarray}
$\{x_k\}$ was a one-sided moving average process and $h$ a polynomial. Central and non-central limit theorems for non-linear functionals of Gaussian fields were explored in \cite{breurmajor} and \cite{dobrushinmajor} respectively. These works used the fact that the weak limit of the normalized sums $S_{n,h}(t)$ is dictated by the Hermite rank of function $h$, which was first shown by Taqqu \cite{taqqu}. Analysis of (\ref{sfloort}) for the Gaussian LRD was explored in \cite{Surglrdappell} and \cite{avramtaqqu} by replacing the Hermite rank with the Appell rank. Vai\v ciulis \cite{vaiciulis} and Surgailis \cite{Surgemp} later investigated (\ref{sfloort}) under the combination of LRD and HT, but products of linear processes were not considered. Thus to the authors' knowledge, central and non-central limit theorems for arbitrary products of two sided linear processes under both LRD and HT have not yet been established, and is a topic worthy of further research. (See also \cite{KOSA} for consideration of the case $s=2$.)

\begin{remark}\label{mikesamira}
Taking $s=2$ in Theorem \ref{mostgen} gives us \cite[Theorem~3]{KOSA} as a corollary. There is a minor miscalculation in the second-last line (Line 17) of \cite[Page~362]{KOSA}. The term $\sum_{l=j+1}^{k+T}c_{j-l}c_{k-l}$ in Line 16 was erroneously taken to be smaller than $\ (j-k)^{-2\sigma}T^{2-2\sigma}$ instead of $\ (j-k)^{1-2\sigma}$. This miscalculation can be corrected by applying Lemma \ref{inequality} (with $\gamma=\sigma$) in Appendix \ref{apptech} of our paper, to Line 15 of \cite{KOSA}, to obtain their results. Also, Kouritzin and Sadeghi \cite[Remark~2]{KOSA} mention that the constraints for handling LRD and those for HT \emph{decouple}, which they explain through the structure of the terms $d_k$. This decoupling phenomenon is observed in our proof as well.
\end{remark}

\begin{remark}\label{sigmasless1}
	Since $\sigma_r \in (\frac12,1],\ \alpha_i \in (1, \infty)$, there exists $\epsilon, \overline\epsilon > 0$ such that $\sigma_r - \epsilon \in \left(\frac12,1\right)$ and $\alpha_i - \overline\epsilon\in (1,2)\cup (2,\infty)$. It can be checked that (Tail, Decay) also hold for $\alpha_i - \overline\epsilon \mbox{ and }\sigma_r - \epsilon$ instead of $\alpha_i \mbox{ and }\sigma_r$ respectively. Thus, by a limit argument, it suffices to assume that $\sigma_r \in (\frac12,1),$ and $\alpha_i \in (1,2)\cup (2,\infty)$. Also, (Decay) implies that $\left|c_l^{(r)}\right| \stackrel{l}\ll \left\{ \begin{array}{ll} 1 & l=0\\|l|^{-\sigma_r}& l\neq 0\end{array}\right.$. The proof of Theorem \ref{mostgen} only differs cosmetically from the notationally simpler case where $\xi_l^{(1)} = \hdots = \xi_l^{(s)} = \xi_l$, and $\ \sigma_1 = \hdots = \sigma_s = \sigma$, hence we can further assume that $c_l^{(1)} = \hdots = c_l^{(s)} = c_l$. Throughout the paper, we only prove this later case, and provide Remark \ref{difffromgen} concerning the notational changes that would have to be made to prove the case where the innovations and LRD coefficients are allowed to be unequal.
\end{remark}
\begin{remark}\label{htpreclude}
	The following calculation will illustrate why we consider the case $\alpha_0>2$ in (Tail) to not possess heavy tails, and the case $\alpha_0\in(1,2]$ to have possible heavy tails. If $\alpha_0 > 2$, then $\alpha_i = \frac{s}{s-i}\alpha_0 > 2\ $ for $i\in\{0,1,\hdots,\lfloor\frac{s-1}{2}\rfloor\}$. 
When $\pi$ is a permutation of $\{1,2,\hdots,s\}$, we see from (Tail), that $\forall\ 0\leq i\leq \left\lfloor\frac{s-1}{2}\right\rfloor$,
	\begin{eqnarray}\nonumber\label{lighttail}
	&&E\left[\prod_{r\in\{\pi(1),\hdots,\pi(s-i)\}}\left|\xi_1^{(r)}\right|^2\right]
	\\ \nonumber
	&=& 2\int_0^{\infty}tP\left(\prod_{r\in\{\pi(1),\hdots,\pi(s-i)\}}\left|\xi_1^{(r)}\right|>t\right)\ dt
	\\
	&\ll&
	2\!\int_0^{1} 1\ dt +\ 2\!\int_1^{\infty}t^{1-\alpha_i}\ dt\ \ \ll\ \ 2 + \frac{2}{\alpha_i-2}\ \ <\ \ \infty.
	\end{eqnarray}
	We conclude that $E\left[\prod_{r=1}^{s}\!\left(\!1+\left(\xi_1^{(r)} \right)^2 \right) \right]<\infty$, which precludes heavy tails.
\end{remark}

Our second main result is a multivariate version of Theorem \ref{mostgen}. 
This theorem follows from linearity of limits and Theorem \ref{mostgen}. 
\begin{theorem}\label{multiv}
Let $s \in \mathbb{N},\ \alpha_0 > 1,\ \alpha_i =\frac{s}{s-i}\alpha_0\ \mbox{ for}\ 1 \leq i \leq \left\lfloor \frac{s-1}{2} \right\rfloor$ and\\
$\left\{ \left(\Xi_{l}^{(1)},\Xi_{l}^{(2)},\hdots,\Xi_{l}^{(s)}\right)\right\}_{l\in\mathbb{Z}}\ $ be i.i.d.\ zero-mean random matrices in $\mathbb{R}^{m\times s}$, such that $\ E \left[\left\|\Xi_1^{(r)}\right\|_F^{s\vee 2}\right]<\infty,\ \ \forall\ 1\leq r\leq s$, and
	\begin{eqnarray}\nonumber
	\max_{\pi\in\mathcal P_s}\ \max_{1 \leq i \leq \left\lfloor \frac{s-1}{2} \right\rfloor}\ \sup_{t\geq 0}\ t^{\alpha_i}P\left(\prod_{r\in\{\pi(1),\hdots,\pi(s-i)\}}\left\|\Xi_1^{(r)}\right\|_F > t\!\right) < \infty\ .
	\end{eqnarray}
	Moreover, let $\ \mathbb{R}^{d\times m}$-valued matrices $\ \left\{\left(C_{l}^{(1)},C_{l}^{(2)},\hdots,C_{l}^{(s)}\right)\right\}_{l\in\mathbb{Z}}\ $ satisfy\\
	$\sup_{l\in\mathbb Z}\ |l|^{\sigma_r} \left\|C_l^{(r)}\right\|_F<\infty\ $, for some $\sigma_r \in \left(\frac12,1\right]$.
	For$\ 1\leq r\leq s,\ k \in \mathbb Z,\ $ define $\ X_{k}^{(r)}=\sum\limits_{l=-\infty }^{\infty}C_{k-l}^{(r)} \Xi_{l}^{(r)}$. Then, $\ \lim\limits_{n\rightarrow \infty}{n^{-\frac{1}{p}}}\sum\limits_{k=1}^{n}\left(\substack{s\\\bigotimes\\{r=1}} X_k^{(r)} - E\left[\substack{s\\\bigotimes\\{r=1}} X_k^{(r)}\right]\right)=0\ \ \mbox{a.s,}\ \ $for the values of $p$ as in (\ref{cond4}).
\end{theorem}

We illustrate Theorem \ref{multiv} by considering the simple case, $s=d=m=2$. Thus we can express,
$$\Xi_l^{(r)} =
\begin{bmatrix}
\xi_{l,1}^{(r)}\\
\xi_{l,2}^{(r)}
\end{bmatrix},\ C_l^{(r)} =
\begin{bmatrix}
c_{l,11}^{(r)} & c_{l,12}^{(r)}\\
c_{l,21}^{(r)} & c_{l,22}^{(r)}
\end{bmatrix},\ X_k^{(r)} =
\begin{bmatrix}
x_{k,11}^{(r)} + x_{k,12}^{(r)}\\
x_{k,21}^{(r)} + x_{k,22}^{(r)}
\end{bmatrix},
$$
where, $x_{k,ij}^{(r)} = \sum\limits_{l=-\infty }^{\infty}c_{k-l,ij}^{(r)}\xi_{l,j}^{(r)}$. Since $s=2$, that gives us for all $1\leq i,j\leq 2$, that
\begin{align}\nonumber\label{fourterms}
    \left(\substack{s\\\bigotimes\\{r=1}} X_k^{(r)}\right)_{ij}\ =&\ \left(x_{k,i1}^{(1)} + x_{k,i2}^{(1)}\right)\left(x_{k,j1}^{(2)} + x_{k,j2}^{(2)}\right)\\
    =&\ x_{k,i1}^{(1)}x_{k,j1}^{(2)} + x_{k,i1}^{(1)}x_{k,j2}^{(2)} + x_{k,i2}^{(1)}x_{k,j1}^{(2)} + x_{k,i2}^{(1)}x_{k,j2}^{(2)}\ .
\end{align}

Let us consider the first term in the right hand side of (\ref{fourterms}). Using Theorem \ref{mostgen} with $s=2$ on $d_k=x_{k,i1}^{(1)}x_{k,j1}^{(2)}$, we get that
$$\lim\limits_{n\rightarrow \infty}{n^{-\frac{1}{p}}}\sum\limits_{k=1}^{n}\left(x_{k,i1}^{(1)}x_{k,j1}^{(2)} - E\left[x_{k,i1}^{(1)}x_{k,j1}^{(2)}\right]\right)=0\ \ \mbox{a.s,}\ \ $$
for the values of $p$ as in (\ref{cond4}). A similar MSLLN holds for the rest of the terms in (\ref{fourterms}) for the same values of $p$, hence by linearity of limits we get
$$\lim\limits_{n\rightarrow \infty}{n^{-\frac{1}{p}}}\sum\limits_{k=1}^{n}\left(\substack{s\\\bigotimes\\{r=1}} X_k^{(r)} - E\left[\substack{s\\\bigotimes\\{r=1}} X_k^{(r)}\right]\right)_{ij}=0\ \ \mbox{a.s.}$$
This holds for all $1\leq i,j\leq 2$, and thus see that Theorem \ref{multiv} is true in this case.

\section{Proof of Theorem \ref{mostgen}}\label{chap4}
\subsection{Light-tailed Case of Theorem \ref{mostgen}} 
\noindent
Keeping Remarks \ref{sigmasless1} and \ref{htpreclude} in mind, we first present a theorem that handles long-range dependence under the condition $\alpha_0 > 2$.
\begin{theorem}\label{longonlygen}
Let $E \left[(\xi_1)^{2s} \right]<\infty$, $d_k$ be defined as in (\ref{dkdef}), $d=E[d_1]$, and Condition (decay) hold.
Then, $\ \lim\limits_{n\rightarrow \infty}{n^{-\frac{1}{p}}}\sum\limits_{k=1}^{n}\left(d_{k}-d\right)=0\ \ \mbox{a.s.}\ \ $for
	\begin{eqnarray}\label{longonlyp}
	0\ \ <\ \ p\ \ <\ \ \left\{ \begin{array}{ll} 2 \wedge \frac{1}{2 - 2\sigma}, & s=2 \\ \frac{2}{3 - 2\sigma}, & s \neq 2 \end{array}\right..
	\end{eqnarray}
Furthermore, if $E[(\xi_1)^\chi]=0$ for all odd $0<\chi<s$ and $s$ is even, then the constraint for (\ref{longonlyp}) can be relaxed to 
	\begin{eqnarray}\label{longonlyp2}
	0\ \ <\ \ p\ \ <\ \ 2 \wedge \frac{1}{2 - 2\sigma}\ .
	\end{eqnarray}
\end{theorem}
\begin{proof}
	By expanding the expressions for $d_k$ and $d$, we get that
	\begin{eqnarray}\nonumber
	\ \ \sum_{k=1}^n (d_k-d)\ = \ \sum_{k=1}^n\sum\limits _{l_1=-\infty }^{\infty}\hdots\sum\limits _{l_s=-\infty }^{\infty} \left(\prod_{r=1}^{s}c_{k-l_r}\right) \left(\prod_{r=1}^{s}\xi_{l_r} - E\left(\prod_{r=1}^{s}\xi_{l_r}\right)\right).
	\end{eqnarray}
	This expression for $\sum_{k=1}^n (d_k-d)$ can be broken up in several sums based on the combinations of subscripts of $\xi$'s that are equal. That is, $\sum_{k=1}^n (d_k-d)$ can be seen as the sum of
	\begin{eqnarray}\label{sqlambdaq}
	\!\!\!\!\!\!\!\!\!\!S_n(q, \lambda_q)\ =\ \sum_{k=1}^{n}\ \ \sum_{l_1\neq l_2\neq \hdots \neq l_q} \left(\prod_{r=1}^{q} c_{k-l_r}^{a_r}\right) \left(\prod_{r=1}^{q}\xi_{l_r}^{a_r} - E\left(\prod_{r=1}^{q}\xi_{l_r}^{a_r}\right)\right),
	\end{eqnarray}
	where $q$ ranges over $\{1,2,\hdots,s\}$, and $\lambda_q = (a_1, a_2, \hdots, a_q)$ is a decreasing partition of $s$, i.e. it satisfies $a_1 + \hdots + a_q = s$ and $a_1 \geq a_2 \geq \hdots \geq a_q\geq 1$. We will now work with an analogous summation $Y_{n',n,\delta}^{\lambda_q}$, with general random variables $\psi_{l}^{(r)}$ instead of $\xi_{l}^{a_r}$.
	
	\subsubsection{Bounding covariance of $\prod_{r=1}^{q}\psi_{l_{r}}^{(r)}$ and $\prod_{r=1}^{q}\psi_{m_{r}}^{(r)}$}
	\noindent
	We first give the following definitions.
	\begin{definition}\label{matching}
		For $\ q \in \mathbb{N},\ v\in\{1,2,\hdots,q\}$, let the sets $V_r = V_r^{v,q}$ for $1\leq r\leq 6$, be such that $V_1,V_2,V_3$ partition $\{q-v+1,\hdots,q\}$, and $V_4,V_5,V_6$ partition $\{1,\hdots,q-v\}$. A function $\nu=\nu^{q,v}(V_2,V_3,V_4,V_5)$, such that
		$$\nu: V_2\cup V_3\cup V_4\cup V_5\ \rightarrow\ \{1,\hdots,q\},$$
		$\nu$ is injective, $\nu(V_2\cup V_4)\subseteq\{q-v+1,\hdots,q\}$, and $\nu(V_3\cup V_5)\subseteq\{1,\hdots,q-v\}$, will be called a matching function. For ease of notation, we further define $W_1=W_1^{q,v}(\nu)=\ \{q-v+1,\hdots,q\}\setminus\nu(V_2\cup V_4)$, $\ \ W_r=W_r^{q,v}(\nu)=\nu(V_r)\ $ for $2\leq r\leq 5$, and $W_6=W_6^{q,v}(\nu)=\{1,\hdots,q-v\}\setminus\nu(V_3\cup V_5)$.
	\end{definition}
	\begin{remark}\label{vobserve}
		In Definition \ref{matching}, observe that $\ \left|V_1\right|+\hdots+\left|V_6\right|=\left|W_1\right|+\left|\nu(V_2)\right|+\hdots+\left|\nu(V_5)\right|+\left|W_6\right|=q$. Also, since $V_1,V_2,V_3\ $ partition $\{q-v+1,\hdots,q\}$, as do $\ W_1,\nu(V_2),\nu(V_4)$, we get that $\ \left|V_1\right|+\left|V_2\right|+\left|V_3\right|=\left|W_1\right|+\left|\nu(V_2)\right|+\left|\nu(V_4)\right|=v$. Similarly, $\ \left|V_4\right|+\left|V_5\right|+\left|V_6\right|=\left|\nu(V_3)\right|+\left|\nu(V_5)\right|+\left|W_6\right|=q-v$. Finally, due to injectivity of $\nu$, we have $|\nu(V_r)|=|V_r|$ for $2\leq r\leq 5$.
	\end{remark}
	\begin{definition}\label{sets}
		Let $\ q \in \mathbb{N},\ v\in\{1,2,\hdots,q\}$, and $\Delta=\Delta_q$ be the set of all tuples in $\mathbb{Z}^q$ with distinct elements, i.e. $\boldsymbol\ell\in\Delta$ satisfies $l_i\neq l_j$ for all $1\leq i<j\leq q$\footnote[1]{As mentioned in Subsection \ref{list}, for $\boldsymbol\ell\in\mathbb{Z}^d$, $l_i$ denotes the $i$th coordinate of $\boldsymbol\ell$, where $1\leq i\leq q$.}. For sets $V_1,...,V_6$ and matching function $\nu$ as in Definition \ref{matching}, we let
		\begin{eqnarray}\nonumber
		\Delta\!\times\!\Delta(V_1,...,V_6,\nu) = \{(\boldsymbol\ell,\textit{\textbf{m}})\in\Delta\!\times\!\Delta:\ l_{r}=m_{\nu(r)},\ \forall\ r\in V_2\cup V_3\cup V_4\cup V_5\}.
		\end{eqnarray}
		Observe that the collection $\left\{\Delta\times\Delta(V_1,...,V_6,\nu):\{V_1,V_2,V_3\} \mbox{ partitions }\right.$\\$\{q\!-\!v\!+\!1,\hdots,q\},\ \{V_4,V_5,V_6\} \mbox{ partitions } \{1,\hdots,q\!-\!v\},\ \nu=\nu^{q,v}(V_2,V_3,V_4,V_5)$ is a matching function$\left.\right\}$ partitions $\Delta\times\Delta$ .
	\end{definition}
	The following lemma bounds the covariance of $\prod_{r=1}^{q}\psi_{l_r}^{(r)}$ and $\prod_{r=1}^{q}\psi_{m_{r}}^{(r)}$.
	\begin{lemma}\label{covlemma}
		Let $\ q \in \mathbb{N},\ v\in\{1,2,\hdots,q\},\ \delta\geq 1$, and $\{(\psi_l^{(1)},\hdots,\psi_l^{(q)})\}_{l\in \mathbb{Z}}$ be i.i.d. $\ \mathbb{R}^q$-valued random vectors, such that
		\begin{eqnarray}\label{covassump}
		\quad\left\{\begin{array}{l} E\left(\psi_1^{(r)}\right)\qquad \ll\  \textbf{1}_{\{1\leq r\leq q-v\}},
		\\
		E\left[\left(\psi_1^{(r)}\right)^2\right]\ \ \ll\ \delta\textbf{1}_{\{r=1\}}+\textbf{1}_{\{r\neq 1\}},
		\end{array}\right.\qquad \forall\ \ 1\leq r \leq q\ .
		\end{eqnarray}
		Then, for q, v and $(\boldsymbol\ell,\textit{\textbf{m}})\in\Delta\times\Delta(V_1,...,V_6,\nu)$ as in Definition \ref{sets},
		\begin{eqnarray}\nonumber
		&&\left|E\left(\prod_{r=1}^{q}(\psi_{l_r}^{(r)}\psi_{m_r}^{(r)})\right) - E\left(\prod_{r=1}^{q}\psi_{l_r}^{(r)}\right)E\left(\prod_{r=1}^{q}\psi_{m_r}^{(r)}\right)\right|
		\\
		&\stackrel{\delta}{\ll}&
		\left\{ \begin{array}{ll} 0,\quad &\left|V_1\right|>0\ \mbox{ or }\ \left|W_1\right|>0\ \mbox{ or }\ \left|V_6\right|=q,\\ 1,\quad &0<\left|V_6\right|<q,\ \left|V_1\right|=\left|V_4\right|=\left|V_5\right|=\left|W_1\right|=0,\\ \delta,\quad &\mbox{otherwise.}\end{array}\right.
		\end{eqnarray}
	\end{lemma}
	\begin{proof}
The first equation in (\ref{covassump}) tells us that $\left\{\psi_{l}^{(r)}\!,\ r\in\{q-v+1,\hdots,s\},l\in\mathbb{Z}\right\}$ are zero mean 
and they will be referred to as the zero-mean $\psi$'s. 
The second equation in (\ref{covassump}) says that $\left\{\psi_{l}^{(1)}\right\}$ may have distinctly different second moments than $\left\{\psi_{l}^{(r)},\ r>1\right\}$, which is important because we will substitute different values in place of $\left\{\psi_{l}^{(1)}\right\}$. (\ref{covassump}) will also come up as (\ref{psiassump}) in Lemma \ref{psilemma}. When $V_1\cup V_2\cup V_3\neq \phi$, due to the independence of $\psi$'s with different subscripts, and the zero-mean property of $\psi_{l_r}^{(r)}$ for $r\in V_1\cup V_2\cup V_3$ in (\ref{covassump}), we have
		\begin{eqnarray}\nonumber
		E\left(\prod_{r=1}^{q}\psi_{l_r}^{(r)}\right) = E\left(\prod_{r\in V_4\cup V_5\cup V_6}\psi_{l_r}^{(r)}\right)\left(\prod_{r\in V_1\cup V_2\cup V_3}E\left(\psi_{l_r}^{(r)}\right)\right)=0\ .
		\end{eqnarray}
		Similarly, when $W_1\cup \nu(V_2)\cup \nu(V_4)\neq \phi$, we get that $E\left(\prod_{r=1}^{q}\psi_{m_r}^{(r)}\right)=0$. Hence, when $V_1\cup V_2\cup V_3\neq \phi\ $ or $W_1\cup \nu(V_2)\cup \nu(V_4)\neq \phi$, we get that
		\begin{eqnarray}\label{meanzero}
		E\left(\prod_{r=1}^{q}\psi_{l_r}^{(r)}\right)E\left(\prod_{r=1}^{q}\psi_{m_r}^{(r)}\right)=0\ .
		\end{eqnarray}\\
		\textbf{Case 1:} $\left|V_1\right|>0\ \mbox{ or }\ \left|W_1\right|>0\ \mbox{ or }\ \left|V_6\right|=q$.\\
		This case deals with situations when there is at least one unmatched zero-mean $\psi$, or when all $\psi$'s are unmatched. $\left|V_1\right|>0$ and $\left|W_1\right|>0$ imply (\ref{meanzero}) holds. When $V_1\neq \phi$, we see from Definition \ref{matching}, that for all $r\in V_1$, $l_r\neq m_j$ for all $1\leq j\leq q$. Hence, due to the independence of $\psi$'s with different subscripts, and the zero-mean property of $\psi_{l_r}^{(r)}$ for $r\in V_1$, we get that
		\begin{eqnarray}\label{ebcum0}
		\!\!\!\!\!\!\!\!\!\!\!\!\!\!E\left(\prod_{r=1}^{q}(\psi_{l_r}^{(r)}\psi_{m_r}^{(r)})\right) = E\left(\prod_{r\in\{1,\hdots,q\}\setminus V_1}\!\!\psi_{l_{r}}^{(r)}\ \prod_{r=1}^{q}\psi_{m_r}^{(r)}\!\right)\prod_{r\in V_1}E\left(\psi_{l_r}^{(r)}\!\right) = 0\ .
		\end{eqnarray}
		Similarly, (\ref{ebcum0}) holds when $W_1\neq \phi$. Thus, when $\left|V_1\right|>0\ \mbox{or}\ \left|W_1\right|>0$, from (\ref{meanzero}) and (\ref{ebcum0}), we get that
		\begin{eqnarray}\label{ces1}
		\left|E\left(\prod_{r=1}^{q}(\psi_{l_r}^{(r)}\psi_{m_r}^{(r)})\right) - E\left(\prod_{r=1}^{q}\psi_{l_r}^{(r)}\right)E\left(\prod_{r=1}^{q}\psi_{m_r}^{(r)}\right)\right|=0\ .
		\end{eqnarray}
		When $\left|V_6\right|=q$, we must have $v=0$ and none of the $l$'s are equal to any of the $m$'s, i.e. $\{l_1,\hdots,l_q\}\cap\{m_1,\hdots,m_q\}=\phi$. In that scenario, due to the independence of $\psi_{l_r}^{(r)}$'s with $\psi_{m_r}^{(r)}$'s, (\ref{ces1}) holds as well.\\\\
		\textbf{Case 2:} $0<\left|V_6\right|<q,\ \left|W_1\right|=\left|V_1\right|=\left|V_4\right|=\left|V_5\right|=0$.\\
		In this case we will show that $l_1\not\in \{m_1,\hdots,m_q\}$ and $m_1\not\in \{l_1,\hdots,l_q\}$, i.e. $\psi_{l_{1}}^{(1)}$ and $\psi_{m_{1}}^{(1)}$ will remain unmatched, so we do not have to deal with the second moment of $\psi^{(1)}$. From Remark \ref{vobserve}, note that $\left|V_4\right|+\left|V_5\right|+\left|V_6\right|=q-v$, hence $0< \left|V_6\right|<q$ along with $\left|V_4\right|=\left|V_5\right|=0$ implies that $0<v<q$. Since $v$ is the cardinality of $V_1\cup V_2\cup V_3$, this means that $\{1,\hdots,q\}\neq V_1\cup V_2\cup V_3\neq \phi$, and (\ref{meanzero}) holds in this case.
		
		From Remark \ref{vobserve}, using injectivity of $\nu$, we get that $\ \left|V_1\right|+\left|V_2\right|+\left|V_3\right|=\left|W_1\right|+\left|V_2\right|+\left|V_4\right|$. Thus, $\left|V_1\right|=\left|W_1\right|=0$ implies that $\left|V_3\right|=\left|V_4\right|$. Also, $v<q$ implies that $q-v\geq 1$, hence $1\in V_4\cup V_5\cup V_6$ and $1\in \nu(V_3)\cup \nu(V_5)\cup W_6$. Further, $\left|V_3\right|=\left|V_4\right|=\left|V_5\right|=0$ ensures that $1\in V_6$ and $1\in W_6$. This means that $l_1\not\in \{m_1,\hdots,m_q\}$ and $m_1\not\in \{l_1,\hdots,l_q\}$. Hence, due to independence of $\psi$'s with unequal subscripts, Cauchy-Schwartz inequality, and (\ref{covassump}), we find
		\begin{eqnarray}\nonumber\label{ebcum1}
		E\left(\prod_{r=1}^{q}(\psi_{l_r}^{(r)}\psi_{m_r}^{(r)})\right)\! &=&\! E\left(\psi_{l_1}^{(1)}\right)E\left(\psi_{m_1}^{(1)}\right)E\left(\prod_{r=2}^{q}(\psi_{l_r}^{(r)}\psi_{m_r}^{(r)})\right)
		\\ \nonumber
		\!&\leq&\!
		E\!\left(\psi_{l_1}^{(1)}\right)\!E\!\left(\psi_{m_1}^{(1)}\right)\!\left(\prod_{r=2}^{q}\!E\left[\left(\psi_{l_r}^{(r)}\right)^2\right]\prod_{r=2}^{q}\!E\left[\left(\psi_{m_r}^{(r)}\right)^2\right]\right)^{\!\!\frac12}
		\\
		\!&\stackrel{\delta}{\ll}&\!1\ .
		\end{eqnarray}
		From (\ref{meanzero}) and (\ref{ebcum1}), we get that
		\begin{eqnarray}\label{ces2}
		\left|E\left(\prod_{r=1}^{q}(\psi_{l_r}^{(r)}\psi_{m_r}^{(r)})\right) - E\left(\prod_{r=1}^{q}\psi_{l_r}^{(r)}\right)E\left(\prod_{r=1}^{q}\psi_{m_r}^{(r)}\right)\right|\ \stackrel{\delta}{\ll}\ 1\ .
		\end{eqnarray}\\
		\textbf{Case 3:} None of the above.\\
		For all other cases, we will get various bounds, and we will show that the worst of them is $\delta$. Due to the independence of $\psi$'s with different subscripts, Cauchy-Schwartz inequality, and the fact that $E\left[\left(\psi_1^{(r)}\right)^2\right]\ll\delta$ (from (\ref{covassump})), we have that
		\begin{eqnarray}\nonumber\label{ebcumdelta}
		\hspace{-1cm} E\left(\prod_{r=1}^{q}(\psi_{l_r}^{(r)}\psi_{m_r}^{(r)})\right)\!\!\!
		&\leq&
		\!\!\!\left(\prod_{r=1}^{q}E\left[\left(\psi_{l_r}^{(r)}\right)^2\right]\prod_{r=1}^{q}E\left[\left(\psi_{m_r}^{(r)}\right)^2\right]\right)^{\!\!\frac12}
		\\
		&\stackrel{\delta}{\ll}&
		\!\!\!\left(\delta^2\prod_{r=2}^{q}E\left[\left(\psi_{l_r}^{(r)}\right)^2\right]\prod_{r=2}^{q}E\left[\left(\psi_{m_r}^{(r)}\right)^2\right]\right)^{\!\!\frac12}\ \stackrel{\delta}{\ll}\ \delta.
		\end{eqnarray}
		We also see that $E\left(\prod_{r=1}^{q}\psi_{l_r}^{(r)}\right)E\left(\prod_{r=1}^{q}\psi_{m_r}^{(r)}\right)\ \stackrel{\delta}{\ll}1$, due to independence of $\psi$'s with different subscripts, so using (\ref{ebcumdelta}) and Triangle Inequality, we get that
		\begin{eqnarray}\label{ces3}
		\!\!\!\!\!\!\!\!\!\!\!\!\left|E\left(\prod_{r=1}^{q}(\psi_{l_r}^{(r)}\psi_{m_r}^{(r)})\right) - E\left(\prod_{r=1}^{q}\psi_{l_r}^{(r)}\right)E\left(\prod_{r=1}^{q}\psi_{m_r}^{(r)}\right)\right|\ \stackrel{\delta}{\ll}\ \delta+1\ \stackrel{\delta}{\ll}\ \delta\ .
		\end{eqnarray}\\
		Lemma \ref{covlemma} follows from (\ref{ces1}, \ref{ces2}) and (\ref{ces3}).
	\end{proof}
	
The next lemma bounds the second moment of a class of partial sum differences,
which we will use first to bound the second moment of $S_n(q,\lambda_q)$ and later on to handle heavy tails. The proof is technical and involves repeated applications of Lemmas \ref{inequality}  and \ref{newineq}, and is relegated to the supplementary materials, but follows the idea in Lemma \ref{covlemma} of considering sets corresponding to partitions of $s$.
	\begin{lemma}\label{psilemma}
		Let $n'<n\in \mathbb{N}\cup\{0\}$, $s \in \mathbb{N},\ \delta\geq 1,\ \lambda_q = (a_1, a_2, \hdots, a_q)$ is a decreasing partition of $s$, and $v = \left|\{1\leq r\leq q: a_r = 1\}\right|$. Let $\{c_l\}_{l\in\mathbb Z}$ satisfy $\ \sup\limits_{l\in\mathbb Z}|l|^\sigma|c_l|<\infty,\ \ \mbox{for some}\ \ \sigma \in \left(\frac{1}{2},1\right)$, and $\{(\psi_l^{(1)},\hdots,\psi_l^{(q)})\}_{l\in \mathbb{Z}}$ be i.i.d $\ \mathbb{R}^q$-valued random vectors, such that
		\begin{eqnarray}\label{psiassump}
		\quad\left\{ \begin{array}{l} E\left(\psi_1^{(r)}\right)\qquad \ll\  \textbf{1}_{\{1\leq r\leq q-v\}},
		\\
		E\left[\left(\psi_1^{(r)}\right)^2\right]\ \ \ll\ \delta\textbf{1}_{\{r=1\}}+\textbf{1}_{\{r\neq 1\}},
		\end{array}\right.\qquad \forall\ \ 1\leq r \leq q\ .
		\end{eqnarray}
		\begin{eqnarray}\nonumber
		\hspace{-0.5cm}\mbox{Define},\ Y_{n',n,\delta}^{\lambda_q} = \sum_{k=n'+1}^{n}\  \sum_{\boldsymbol\ell\in\Delta}\left(\prod_{r=1}^{q} c_{k-l_r}^{a_r}\right) \left(\prod_{r=1}^{q}\psi_{l_r}^{(r)} - E\left(\prod_{r=1}^{q}\psi_{l_r}^{(r)}\right)\right).
		\end{eqnarray}
		Then,$\ E\left[(Y_{n',n,\delta}^{\lambda_q})^2\right]\!\stackrel{n',n,\delta}{\ll}
		\!\left\{ \begin{array}{ll}\delta\ (n-n'), &a_q \geq 2,\\ \delta\ (n-n')\ l_{s,\sigma}(n-n'), &a_1 = 1,\\ (\delta\ (n-n'))\vee((n-n')\ l_{1,\sigma}(n-n')),\! &a_q=1,a_1\geq 2,\end{array}\right.$\\\\
		where $\boldsymbol\ell$ and $l_{s,\sigma}$ are defined in the Notation List in Subsection \ref{list}. Further, if s is even and $E\left(\psi_1^{(r)}\right) = 0\ $for odd $a_r$, then this bound can be tightened to
		$$E\left[(Y_{n',n,\delta}^{\lambda_q})^2\right]\ \ \stackrel{n',n,\delta}{\ll}\ \ (\delta\ (n-n'))\ \vee\ ((n-n')\ l_{2,\sigma}(n-n')),$$
		when $a_q = 1$ and $a_1 \geq 2$.
	\end{lemma}
	
	\subsubsection{Rate of Convergence for Theorem \ref{longonlygen}}
	\noindent
Returning to the proof of Theorem \ref{longonlygen}, we will bound the 
second moment of $S_n(q, \lambda_q)$ defined in (\ref{sqlambdaq}). 
In Lemma \ref{psilemma}, taking $\psi_{l_r}^{(r)} = \xi_{l_r}^{a_r}\ $ for $\ 1\leq r\leq q$, and$\ \delta = 1$ (since $E\left[\left(\xi_{l_1}^{a_1}\right)^2\right]\ \stackrel{n',n}{\ll}\ 1$), we see that $Y_{n',n,\delta}^{\lambda_q}$ becomes $S_n(q, \lambda_q)-S_{n'}(q, \lambda_q)$, and
	\begin{eqnarray}\label{res}
	\!E\left[\big(S_n(q, \lambda_q)-S_{n'}(q, \lambda_q)\big)^2\right] \stackrel{n',n}{\ll} \left\{ \begin{array}{ll} n-{n'}, &a_q \geq 2\\ (n-{n'})\ l_{s,\sigma}(n-{n'}), &a_1 = 1\\ (n-{n'})\ l_{1,\sigma}(n-{n'}),\! &a_q=1,a_1\geq 2.\end{array}\right.
	\end{eqnarray}
	But, when $s$ is even and $E\left(\xi_{l}^{a_r}\right)=0$ for odd $a_r$ so $E\left(\psi_{l}^{(r)}\right)=E\left(\xi_{l}^{a_r}\right)=0$, we find from Lemma \ref{psilemma} that (\ref{res}) for $a_q=1$ and $a_1\geq 2$ improves to
	\begin{eqnarray}\label{res1}
	\!\!\!\!\!E\left[\big(S_n(q,\lambda_q)-S_{n'}(q,\lambda_q)\big)^2\right] \stackrel{n',n,\delta}{\ll}(\delta\ (n-{n'}))\vee ((n-{n'})\ l_{2,\sigma}(n-{n'})).
	\end{eqnarray}
The bounds in (\ref{res}) and (\ref{res1}) are given in terms of a partition $\lambda_q$. 
We can check which partitions are possible for a given $s$, and then apply (\ref{res}) and (\ref{res1}) to bound the second moment of $\sum_{k=1}^n (d_k-d)$. Recall that $\ s=a_1+a_2+\hdots+a_q\ $ and $\ a_1\geq a_2\geq \hdots \geq a_q\geq 1$. When $s=1$, none of the cases except $a_1=1$ are possible, and when $s=2$, the third case i.e. $a_q=1,\ a_1\geq 2$ is not possible. Hence, we get from (\ref{res}), that
	\begin{eqnarray}\label{max0}
	E\left[\big(S_n(q, \lambda_q)-S_{n'}(q, \lambda_q)\big)^2\right] 
	&\stackrel{n',n}{\ll} & 
	\left\{ \begin{array}{ll} (n-{n'})\ l_{2,\sigma}(n-{n'}), &s=2\\ (n-{n'})\ l_{1,\sigma}(n-{n'}), &s\neq 2\end{array}\right.,\qquad\quad
	\end{eqnarray}
	and from (\ref{res1}), that if $s$ is even and $\xi_l$ is a symmetric random variable, then
	\begin{eqnarray}\label{max0.5}
	\!\!\!E\left[\big(S_n(q, \lambda_q)-S_{n'}(q, \lambda_q)\big)^2\right]\  \stackrel{n',n}{\ll}\ (n-{n'})\ l_{2,\sigma}(n-{n'})\ .
	\end{eqnarray}
	Let $n_{r}=2^{r}$, $n\in\left[n_{r},n_{r+1}\right)$ and $r\in\mathbb{N}\cup\{0\}$. Then, putting $n=n_r$ and ${n'}=0$ in (\ref{max0}), we get that
	\begin{eqnarray}\label{max1}
	E\left[\big(S_{n_r}(q, \lambda_q)\big)^2\right]\ 
	&\stackrel{r}{\ll} &\ 
	\left\{ \begin{array}{ll} n_r\ l_{2,\sigma}(n_r), &s=2\\ n_r\ l_{1,\sigma}(n_r),\ \ &s\neq 2\end{array}\right..\qquad\qquad
	\end{eqnarray}
	$\bullet\ $ First, consider $s\neq 2$. Then for $n_{r}\leq n' < n < n_{r+1}$, it follows from (\ref{max0}), using Theorem \ref{stout} with $\ Z_i=S_{i}(q,\lambda_q)-S_{i-1}(q,\lambda_q)\ $ and $\ f(n)=n\ l_{1,\sigma}(n)$, that
	\begin{eqnarray}\label{max2}
	E\left[\max_{n_{r}\le n'<n<n_{r+1}}\big(S_n(q, \lambda_q)-S_{n'}(q, \lambda_q)\big)^2\right]\ \stackrel{r}{\ll}\ r^2n_r\ l_{1,\sigma}(n_r)\ .
	\end{eqnarray}
	Combining (\ref{max1}) and (\ref{max2}), we have that
	\begin{eqnarray}\label{combine1}
	\!\!\!\!\!\!\!\!\!\!\!\!\!\!\!\sum_{r=0}^\infty E\left[\max_{n_{r}\le n<n_{r+1}}\left({n^{-\frac{1}{p}}}S_n(q, \lambda_q)\right)^{2}\right]
	&\ll& \sum\limits_{r=0}^\infty r^2n_r^{1-\frac2p}\ l_{1,\sigma}(n_r)\ <\ \infty,
	\end{eqnarray}
	provided $(3-2\sigma)<\frac{2}{p}$, i.e. $p<\frac{2}{3-2\sigma}$. From (\ref{combine1}), it follows by Fubini's Theorem and $n^{\mbox{th}}$ term divergence that $\lim_{n\rightarrow\infty}{n^{-\frac{1}{p}}}S_n(q, \lambda_q)=0\quad a.s.$, for
	\begin{eqnarray}\label{sconv1}
	p<\frac{2}{3-2\sigma}.
	\end{eqnarray}
	$\bullet\ $ Now let $s=2$. Then, using (\ref{max0}) and proceeding along the lines of (\ref{max0.5}-\ref{sconv1}), we get that $\lim_{n\rightarrow\infty}{n^{-\frac{1}{p}}}S_n(q, \lambda_q)=0\quad a.s.$, for
	\begin{eqnarray}\label{sconv2}
	p<2\wedge \frac{1}{2-2\sigma}.
	\end{eqnarray}
	$\bullet\ $ Finally, we consider the case where s is even, and $E[(\xi_1)^\chi]=0$ for all odd $0<\chi<s$. Again, using (\ref{max0}) and proceeding along the lines of (\ref{max0.5}-\ref{sconv1}), we get that $\lim_{n\rightarrow\infty}{n^{-\frac{1}{p}}}S_n(q, \lambda_q)=0\quad a.s.$, for
	\begin{eqnarray}\label{sconv3}
	p<2\wedge \frac{1}{2-2\sigma}.
	\end{eqnarray}
Since $\sum_{k=1}^n (d_k-d)$ is the sum of $S_n(q,\lambda_q)$ over all $q\in\{1,\hdots,s\}$ and partitions $\lambda_q$ (which are finite in number), we get from (\ref{sconv1},\ref{sconv2}) and (\ref{sconv3}), that
	$$\lim_{n\rightarrow\infty}{n^{-\frac{1}{p}}}\sum_{k=1}^n (d_k-d)=0\quad a.s.$$
for the values of $p$ described in (\ref{longonlyp}) and (\ref{longonlyp2}). 
This proves Theorem \ref{longonlygen}.
\end{proof}

\subsection{Heavy-Tailed Case of Theorem \ref{mostgen}}\label{heavytail}
\noindent
We first present three remarks before analyzing the heavy-tailed scenario.
\begin{remark}\label{alphanot2}
From Condition (Tail), we find that heavy tails can only arise when $0 \leq i \leq \lfloor \frac{s-1}{2} \rfloor\ $, i.e. for products of at least $s-\lfloor\frac{s-1}{2}\rfloor=\lceil\frac{s+1}{2}\rceil\ $ terms. 
When $s=1$, Condition (Reg) along with Remark \ref{sigmasless1} eliminate the possibility of heavy tails. When $s\geq 2$, we can assume without loss of generality, that $\alpha_i \in (1,2)\cup (2,\infty)$ (due to Remark \ref{sigmasless1}). However, if $\alpha_i>2$, we see from Remark \ref{htpreclude} that heavy tails do not arise. Since we will deal only with those terms exhibiting heavy tails in this section, we assume that $s\geq 2\ $, $i\in \left\{0,1,\hdots,\left\lfloor\frac{s-1}{2}\right\rfloor\right\}$, and $1 < \alpha_i < 2$.
\end{remark}
\begin{remark}
	For a given partition $\lambda_q=\{a_1,a_2,\hdots,a_q\}$, heavy tails can only come up in the innovation involving the highest power, i.e. $\xi_l^{a_1}$. This is because for a term to possess heavy tails, its variance must be infinite, hence $a_1>\frac{s}{2}$. But that would force the rest of the $a_r$'s to be less than $\frac{s}{2}$, thus precluding heavy tails in terms involving $\xi_l^{a_r}$ for $r\in\{2,\hdots,q\}$. This shows that heavy tails concerning $\alpha_i$ will arise only in the sum
	\begin{eqnarray}\label{sstar}
	\!\!\!\!\!\!\!\!\!S_{n}^\star(i)=\sum_{k=1}^{n}\! \sum_{\substack{l_1,l_2,\hdots,l_{i+1}\\l_1\not\in\{l_2,\hdots,l_{i+1}\}}}\!\!\!\!\!\left(\!c_{k-l_1}^{s-i}\!\prod_{r=2}^{i+1}\!c_{k-l_r}\!\right)\! \left(\!\xi_{l_1}^{s-i}\prod_{r=2}^{i+1}\xi_{l_r} - E\!\left(\!\xi_{l_1}^{s-i}\!\prod_{r=2}^{i+1}\xi_{l_r}\!\right)\!\right)\!.
	\end{eqnarray}
\end{remark}
\begin{remark}
	Alternatively, for heavy tails involving $\alpha_i$, we could also consider the sum $S_{n}(q, \lambda_q)$ (from (\ref{sqlambdaq})) with $a_1 = s-i$, i.e.
	\begin{eqnarray}\nonumber
	S_n(q, \lambda_q) = \sum_{k=1}^{n}\ \sum_{l_1\neq l_2\neq \hdots \neq l_q}\!\!\left(\!c_{k-l_1}^{s-i}\prod_{r=2}^{q} c_{k-l_r}^{a_r}\!\right)\! \left(\!\xi_{l_1}^{s-i}\prod_{r=2}^{q}\xi_{l_r}^{a_r} - E\!\left(\!\xi_{l_1}^{s-i}\prod_{r=2}^{q}\xi_{l_r}^{a_r}\!\right)\right),
	\end{eqnarray}
	where $\lambda_q=(s-i,a_2,\hdots,a_q)$. In fact, note that $S_{n}^\star$ (from (\ref{sstar})) is the sum of $S_{n}(q, \lambda_q)$ over all $q$, and all partitions $\lambda_q$ with $a_1 = s-i$. 
Both $S_{n}^\star(i)$ and $S_{n}(q, \lambda_q)$ have advantages. 
While $S_{n}^\star(i)$ has the advantage of having only one $\xi_l$ with power greater than one, $S_{n}(q, \lambda_q)$ has the advantage of having its summation over $\Delta$, so Lemma \ref{covlemma} can be easily applied to it. Hence, we will mostly use $S_{n}(q, \lambda_q)$ to deal with the truncated terms, and $S_{n}^\star(i)$ for the error terms.
\end{remark}

\subsubsection{Conversion to continuous random variables}
\noindent
Recall that in this section, $i\in \left\{0,1,\hdots,\left\lfloor\frac{s-1}{2}\right\rfloor\right\}$ is fixed. 
We first replace $\xi_l^{s-i}$ with continuous random variables $\zeta_{l}$
to ensure the truncation below does not take place at a point with positive probability. 
Let $\{U_l\}_{l\in\mathbb Z}$ be independent $[-1,1]$-uniform random variables that are independent of $\{\xi_{l}\}_{l \in \mathbb Z}$.
Then,
\begin{eqnarray}\nonumber
S_n(q, \lambda_q)=A_n(q,\lambda_q) - B_n(q,\lambda_q)\ ,
\end{eqnarray}
where we define,
\small{\begin{eqnarray}\nonumber
	\hspace{-0.7cm}&&A_n(q,\lambda_q)=\sum_{k=1}^{n}\sum_{l_1\neq l_2\neq\hdots \neq l_q}\!\!\!\left(\prod_{r=1}^{q} c_{k-l_r}^{a_r}\!\right)\!\! \left(\!\!\left(\xi_{l_1}^{s-i}\!+ U_{l_1}\right)\!\prod_{r=2}^{q}\xi_{l_r}^{a_r} - E\!\left(\!\!\left(\xi_{l_1}^{s-i}\!+ U_{l_1}\right)\!\prod_{r=2}^{q}\xi_{l_r}^{a_r}\!\right)\!\right)\!,
	\\ \nonumber
	\hspace{-0.7cm}&&B_n(q,\lambda_q)=\sum_{k=1}^{n}\sum_{l_1\neq l_2\neq\hdots \neq l_q}\!\!\!\left(\prod_{r=1}^{q} c_{k-l_r}^{a_r}\!\right)\!\! \left(\!U_{l_1}\prod_{r=2}^{q}\xi_{l_r}^{a_r} - E\!\left(\!U_{l_1}\prod_{r=2}^{q}\xi_{l_r}^{a_r}\!\right)\!\right)\ .
	\end{eqnarray}}
\normalsize\!\!{\bf Note:} 1) When $s$ is even and $a_1$ is odd, $\xi_{l_1}^{a_1}+ U_{l_1}$ will still be symmetric
so we can apply the reduced bound (\ref{res1}) when $a_q=1,\ a_1\ge 2$.\\
2) Heavy tails do not arise in $B_n(q,\lambda_q)$ since $E\left[\left(U_{l_1}\right)^2\right]\ $ is constant. \\
For $B_n(q,\lambda_q)$, we take $\psi_{l_r}^{(r)} = \xi_{l_r}^{a_r}\ \ \forall\ 2\leq r\leq q$, $\psi_{l_1}^{(1)} = U_{l_1}$, and $\delta = 1$, in Lemma \ref{psilemma} to get that $\ Y_{n',n,\delta}^{\lambda_q} = B_n-B_o$. This gives us the same bound as in (\ref{res}). Proceeding along the lines of (\ref{max1}\ -\ \ref{sconv3}), we get that $\lim_{n \rightarrow \infty}{n^{-\frac{1}{p}}}{B_n(q,\lambda_q)}\ =\ 0\ \mbox{a.s.}$ for the values of $p$ as mentioned in the statement of Theorem \ref{longonlygen}. 

Moving to $A_n(q,\lambda_q)$ and defining $\zeta_l = \xi_{l}^{s-i} + U_l$, which is a function of $i$, we note that $\zeta_l$ is a continuous random variable since it is a convolution of two random variables, one of which is absolutely continuous. Also, note that $\zeta_l$ has the same tail probability bound as $\xi_{l}^{s-i}$, since
\begin{eqnarray}\label{offdiagbound}\nonumber
\!\!\!\!\!\!\!\!\!\sup_{t\geq 2}\ t^{\alpha_i} P\left(\left|\zeta_{1} \right| > t \right)
&\leq& \sup_{t\geq 2}\ t^{\alpha_i} P\left(\left|\xi_1^{s-i}\right| > t-1 \right)\\
&\stackrel{}{\ll}& \sup_{t\geq 1}\ \left(\frac{t+1}{t}\right)^{\alpha_i}  t^{\alpha_i} P\left(\left|\xi_1^{s-i}\right| > t \right)\ <\ \infty.\ \ 
\end{eqnarray}
Thus, convergence of $S_n(q, \lambda_q)$ is equivalent to that of
\begin{eqnarray}\nonumber
A_n(q,\lambda_q)\ =\ \sum_{k=1}^{n}\ \sum_{l_1\neq l_2\neq\hdots \neq l_q} \left(\prod_{r=1}^{q} c_{k-l_r}^{a_r}\right) \left(\zeta_{l_1}\prod_{r=2}^{q}\xi_{l_r}^{a_r} - E\left(\zeta_{l_1}\prod_{r=2}^{q}\xi_{l_r}^{a_r}\right)\right).
\end{eqnarray}
Summing over all $q$, and partitions $\lambda_q$ where $a_1 = s-i$, we find that convergence of $S_{n}^\star(i)$ (from (\ref{sstar})) is equivalent to that of
\begin{eqnarray}\label{offdiagexpres}
\!\!\!\!\!\!\!\!\!\!T_n(i)\ =\ \sum_{k=1}^{n} \sum_{\substack{l_1,l_2,\hdots,l_{i+1}\\l_1\not\in\{l_2,\hdots,l_{i+1}\}}}\!\!\!\! \left(\!c_{k-l_1}^{s-i}\prod_{r=2}^{i+1}c_{k-l_r}\!\!\right)\! \left(\!\zeta_{l_1}\prod_{r=2}^{i+1}\xi_{l_r} - E\!\left(\!\zeta_{l_1}\prod_{r=2}^{i+1}\xi_{l_r}\!\right)\!\right).
\end{eqnarray}
\normalsize
\subsubsection{Truncation of $\zeta$ with highest power}
\noindent
We now break each $\zeta$ into truncated and error terms so that 
the second moment of the truncated term is finite, hence handled by 
Theorem \ref{longonlygen}. 
The error term convergence will later be proven using Jensen's, H\"older's and Doob's $L_p$ inequalities as well as Borel-Cantelli Lemma.\\
Let $\kappa>0$. Recall from Remark \ref{alphanot2}, that $1\leq \alpha_i\leq 2$. Using condition (\ref{offdiagbound}), fixing $v_r^+= n_r^\frac{\kappa}{2-\alpha_i}$ (where $n_r=2^r$) for $\ r\in \mathbb{N}\cup\{0\}$, and letting $v_r^- = -v_r^+$, we get
\begin{eqnarray}\label{uplus}
\!\!\!\!\!\!\!\!\left\{ \begin{array}{ll} 2\int_0^{v_r^+}P(\zeta_1 > s)s\ ds\ \stackrel{r}{\ll}\  2\int_0^{v_r^+} s^{-\alpha_i}s\ ds\ \ 
\stackrel{r}{\ll}\ \ n^\kappa_r\\
2\left|\int_{v_r^-}^0 P(\zeta_1 < s)s\ ds\right|\ \stackrel{r}{\ll}\ 2\int_{v_r^-}^0 |s|^{-\alpha_i}|s|\ ds \
\stackrel{r}{\ll} \ \ n^\kappa_r,
\end{array}\quad \forall\ r \in\mathbb{N}\cup\{0\}.\right.
\end{eqnarray}
Next, defining i.i.d random variables $\{\overline\zeta_l^{(r)}\}_{l\in \mathbb{Z}}$ and $\{\tilde{\zeta}_l^{(r)}\}_{l\in \mathbb{Z}}$ by
\begin{eqnarray}\label{zetadef}
\quad\left\{ \begin{array}{l} \overline \zeta_l^{(r)}\ =\ v_r^- \vee \zeta_l \wedge v_r^+\\
\tilde{\zeta}_l^{(r)}\ =\ \zeta_l - \overline \zeta_l^r
\end{array}\right. 
\end{eqnarray}
for $r \in \mathbb{N}$, we call $\overline\zeta_{l}^{(r)}$ the truncated terms and $\tilde\zeta_{l}^{(r)}$ the error terms. 
Observe that $\overline \zeta_l^{(r)}$ and $\tilde{\zeta}_l^{(r)}$ are both functions of $r$. Breaking $\zeta_{l}^{(r)}$ into $\overline\zeta_{l}^{(r)}$ and $\tilde{\zeta}_l^{(r)}$ also helps us break up $A_n(q,\lambda_q)\ $ as $\overline A_n^{(r)}(q,\lambda_q)\ +\ \tilde A_n^{(r)}(q,\lambda_q)$, where
\begin{eqnarray}\nonumber
\overline A_n^{(r)}(q,\lambda_q) &=& \sum_{k=1}^{n}\ \sum_{l_1\neq l_2\neq\hdots \neq l_q}\! \left(\prod_{r=1}^{q} c_{k-l_r}^{a_r}\!\right)\! \left(\overline \zeta_{l_1}^{(r)}\prod_{r=2}^{q}\xi_{l_r}^{a_r} - E\left(\overline \zeta_{l_1}^{(r)}\prod_{r=2}^{q}\xi_{l_r}^{a_r}\!\right)\!\right),
\end{eqnarray}
and $\tilde A_n^{(r)}(q,\lambda_q)$ is obtained by replacing $\overline \zeta_{l_1}^{(r)}$ with $\tilde\zeta_{l_1}^{(r)}$, in $\overline A_n^{(r)}(q,\lambda_q)$. Similarly, $T_n(i)$ (from (\ref{offdiagexpres})) can be broken up as $\overline T_n^{(r)}(i)\ +\ \tilde T_n^{(r)}(i)$, where
\begin{eqnarray}\nonumber
\overline T_n^{(r)}(i)\ =\ \sum_{k=1}^{n} \sum_{\substack{l_1,l_2,\hdots,l_{i+1}\\l_1\not\in\{l_2,\hdots,l_{i+1}\}}} \!\!\!\left(c_{k-l_1}^{s-i}\prod_{r=2}^{i+1}c_{k-l_r}\!\right)\! \left(\overline\zeta_{l_1}^{(r)}\prod_{r=2}^{i+1}\xi_{l_r} - E\left(\overline \zeta_{l_1}^{(r)}\prod_{r=2}^{i+1}\xi_{l_r}\!\right)\!\right),
\end{eqnarray}
and $\tilde T_n^{(r)}(q,\lambda_q)$ is obtained by replacing $\overline \zeta_{l_1}^{(r)}$ with $\tilde\zeta_{l_1}^{(r)}$, in $\overline T_n^{(r)}(q,\lambda_q)$.

\subsubsection{Bounding second moment of truncated terms}
\noindent
Recall that $\zeta_{l},\ \overline \zeta_{l}^{(r)},\ \tilde \zeta_{l}^{(r)},\ A_n(q,\lambda_q),\ \overline A_n^{(r)}(q,\lambda_q),\ \tilde A_n^{(r)}(q,\lambda_q),\ T_n(i),\ \overline T_n^{(r)}(i)$, and $\ \tilde T_n^{(r)}(i)$ are defined in terms of a fixed $i\in \left\{0,1,\hdots,\left\lfloor\frac{s-1}{2}\right\rfloor\right\}$. We now bound the second moments for the truncated terms, $\overline \zeta_{l}^{(r)}$.\\
Using (\ref{offdiagbound},\ref{zetadef}), and the formula
\begin{eqnarray}\label{expform}
E[g(X)]=\int_0^\infty g'(t)P(X>t)\ dt\ \ -\ \ \int_{-\infty}^0 g'(t)P(X<t)\ dt,
\end{eqnarray}
for continuously differentiable function $g$ and random variable $X$, we get that
\begin{eqnarray}\nonumber\label{barzetabound}
E[\overline \zeta_l^{(r)}]\ \ &=&\ \ \int_0^{v_r^+} P(\zeta_l > t)\ dt\ \ -\ \  \int_{v_r^-}^0 P(\zeta_l < t)\ dt
\\
&\leq&\ \ \int_0^{\infty} P(|\zeta_l| > t)\ dt\ \ \leq\ \ E|\zeta_l|\ \ \stackrel{r}{\ll}\ \ 1.
\end{eqnarray}
Also, by (\ref{uplus}) and (\ref{expform}), we have
\begin{eqnarray}\label{barzerabound}\nonumber
\hspace{-1cm}E\left[\left|\overline\zeta_l^{(r)}\right|^2\right]\ &=&\  E\left[|v_r^-\vee\zeta_l\wedge v_r^+|^2\right]
\\
&=&\ 2\int_0^{v_r^+}\! P(\zeta_l > s)s\ ds\ -\ 2\int_{v_r^-}^0\! P(\zeta_l < s)s\ ds\ \ \stackrel{r}{\ll}\ \ n^\kappa_r,
\end{eqnarray}
for all $r \in \mathbb{N}$. We shall now use (\ref{barzetabound}) and (\ref{barzerabound}) to bound the second moment of $\overline A_{n}^{(r)}(q,\lambda_q)$, in terms of $n_r^\kappa$. Recall that $\{\overline\zeta_l^{(r)}\}$ are i.i.d., and $E\left[\left|\overline\zeta_l^{(r)}\right|\right]<\infty$. Hence, taking $\psi_{l_1}^{(1)} = \overline \zeta_{l_1}^{(r)}$, $\psi_{l_r}^{(r)} = \xi_{l_r}^{a_r}\ $ for $\ 2\leq r\leq q$, and $\delta = n_r^\kappa\ $ in Lemma \ref{psilemma}, we see that $Y_{n',n,r}$ becomes $\overline A_n^{(r)}(q,\lambda_q)-\overline A_{n'}^{(r)}(q,\lambda_q)$, and
\begin{eqnarray}\label{neweq}\nonumber
&&E\left[\left(\overline A_n^{(r)}(q,\lambda_q)-\overline A_{n'}^{(r)}(q,\lambda_q)\right)^2\right]
\\
&\stackrel{n, r}{\ll} &
\left\{ \begin{array}{ll} n_r^{\kappa}(n-{n'}), &a_q \geq 2\\ n_r^{\kappa}(n-{n'})\ l_{s,\sigma}(n-{n'}), &a_1 = 1\\ (n_r^{\kappa}(n-{n'}))\vee \big((n-{n'})\ l_{1,\sigma}(n)\big),\ \ &a_q=1,\ a_1\geq 2\end{array}\right..
\end{eqnarray}
Recall that due to Remark \ref{alphanot2}, we have assumed that $s\geq 2$ and $0\leq i\leq \lfloor\frac{s-1}{2}\rfloor$. That gives us, $a_1 = s-i = s-\lfloor\frac{s-1}{2}\rfloor=\lceil\frac{s+1}{2}\rceil\geq 2$, so we discard the case $a_1 = 1$ in (\ref{neweq}). When $s=2$, the third case i.e. $a_q=1,\ a_1\geq 2$ is not possible. Hence, we get from (\ref{neweq}) and the fact that maximum of 2 numbers is upper bounded by their sum, that
\begin{eqnarray}\label{max5}\nonumber
&&E\left[\left(\overline A_n^{(r)}(q,\lambda_q)-\overline A_{n'}^{(r)}(q,\lambda_q)\right)^2\right]
\\
&\stackrel{n, r}{\ll} &
\left\{ \begin{array}{ll} n_r^{\kappa}(n-{n'}), &s=2\\ n_r^{\kappa}(n-{n'})\ + \big((n-{n'})\ l_{1,\sigma}(n-{n'})\big),\ \ &s\neq 2\end{array}\right..\qquad\qquad
\end{eqnarray}
Now, putting $n=n_r=2^{r}$ and ${n'}=0$ in (\ref{max5}), we get
\begin{eqnarray}\label{maxother}
E\left[\big(\overline A_{n_r}^{(r)}(q,\lambda_q)\big)^2\right]\ 
&\stackrel{r}{\ll} &\ 
\left\{ \begin{array}{ll} n_r^{1+\kappa}, &s=2\\ n_r^{1+\kappa}+ (n_r\  l_{1,\sigma}(n_r)),\ \ &s\neq 2\end{array}\right..\qquad\qquad
\end{eqnarray}
$\bullet\ $ Let $s\neq 2$. Then for $n_{r}\leq {n'} < n < n_{r+1}$, it follows from (\ref{max5}) and (\ref{maxother}), using Theorem \ref{stout} with $\ Z_i=\overline A_{i}^{(r)}(q,\lambda_q)-\overline A_{i-1}^{(r)}(q,\lambda_q)\ $ and $\ f(n)=n_r^{\kappa}n\ + \big(n\ l_{1,\sigma}(n)\big)$, that
\begin{eqnarray}\nonumber
E\left[\max_{n_{r}\le n<n_{r+1}}\left(\overline A_{n}^{(r)}(q,\lambda_q)\right)^2\right]\ 
&\stackrel{r}\ll&
\ r^2\big[n_r^{1+\kappa}+ \big(n_r\ l_{1,\sigma}(n_r)\big)\big] ,
\end{eqnarray}
which when summed up over all $q$ and over all partitions $\lambda_q$ with $a_1 = s-i$ (recall that $\ i\in \left\{0,1,\hdots,\left\lfloor\frac{s-1}{2}\right\rfloor\right\}$ is fixed), gives us
\begin{eqnarray}\label{maxtrunc1.5}
E\left[\max_{n_{r}\le n<n_{r+1}}\left(\overline T_n^{(r)}(i)\right)^2\right]\ \stackrel{r}\ll\ r^2\Big[n_r^{1+\kappa}\vee\ \big(n_r\ l_{1,\sigma}(n_r)\big)\Big] ,
\end{eqnarray}
since the sum of two functions is upper bounded by twice their maximum.\\
$\bullet\ $ Now, let $s=2$. Then a similar calculation as in the case $s \neq 2$, gives us that
\begin{eqnarray}\label{maxtrunc2}
E\left[\max_{n_{r}\le n<n_{r+1}}\left(\overline T_n^{(r)}(i)\right)^2\right]\ \stackrel{r}\ll\ r^2n_r^{1+\kappa}.
\end{eqnarray}
$\bullet\ $ Finally, we consider the situation where s is even, and $\xi_l$ is symmetric. Clearly $\xi_{l}^{a_j}$ will be symmetric when $a_j$ is odd, implying that $E\left(\xi_{l}^{a_j}\right)=0$ for odd $a_j,\ 2\leq j\leq q$. Also, since $a_1 = s-i$, we see that $\xi_{l}^{a_1}$ will be symmetric when $a_1$ is odd, implying that both $\zeta_l$ and $\overline \zeta_l^{(r)}$ will be symmetric. Hence, proceeding as in the case $s\neq 2$ again, gives us that
\begin{eqnarray}\label{maxtrunc3}
E\left[\max_{n_{r}\le n<n_{r+1}}\left(\overline T_n^{(r)}(i)\right)^2\right]\ \stackrel{r}\ll\ r^2\Big[n_r^{1+\kappa}\vee\ \big(n_r\ l_{2,\sigma}(n_r)\big)\Big],
\end{eqnarray}
which is clearly an improvement over (\ref{maxtrunc1.5}), since the function $l_{2,\sigma}\ \leq\ l_{1,\sigma}$.\\

\subsubsection{Bounding $\tau$th moment of error terms, $\tau\in(1,\alpha_i)$}
\noindent
Taking $1<z<\alpha_i$, and using our tail probability bound in (\ref{offdiagbound}) along with (\ref{expform}), we have that
\begin{eqnarray}\nonumber
E\left|\left(\tilde \zeta_1^{(r)}\right)^+\right|^z
&=&
z\int_0^\infty s^{z-1}P\left(\zeta_1^{(r)} - (\zeta_1^{(r)} \wedge v_r^+) > s\right)\ ds
\\ \nonumber
&=&
z\int_0^\infty s^{z-1}P\left(\zeta_1^{(r)} > v_r^+ + s\right)\ ds
\\ \nonumber
&\stackrel{r}{\ll}& 
\int_{v_r^+}^\infty (s-v_r^+)^{z-1}s^{-{\alpha_i}}\ ds
\\ \nonumber
&\le& (v_r^+)^{-\alpha_i}\int_{v_r^+}^{2v_r^+}(s-v_r^+)^{z-1}\ ds\ \ +\ \ \int_{2v_r^+}^\infty (s-v_r^+)^{z-\alpha_i-1}\ ds
\\\nonumber
&\stackrel{r}{\ll}& (v_r^+)^{z-{\alpha_i}}\ \ 
\stackrel{r}{\ll}\ \ n_r^{\frac{\kappa(z-{\alpha_i})}{2-\alpha_i}}.
\end{eqnarray}
By symmetry $E\left|\left(\tilde \zeta_1^{(r)}\right)^-\right|^z$ has the same bound so for $1<z<\alpha_i$, we get that
\begin{eqnarray}\label{momboundtilde}
\|\tilde \zeta_1^{(r)}\|_z\ \ \stackrel{r}{\ll}\ \ n_r^{\frac{\kappa(z-{\alpha_i})}{z(2-\alpha_i)}}.
\end{eqnarray}
Now, we explore the convergence rates of $\tilde T_n^{(r)}(i)$. Note that
\begin{eqnarray}\label{betaterm0.5}
\!\!\!\!\!\!\!\tilde T_n^{(r)}(i)\ =\sum_{k=1}^{n} \sum_{\substack{l_1,l_2,\hdots,l_{i+1}\\l_1\not\in\{l_2,\hdots,l_{i+1}\}}} \!\!\!\!\!\left(\!c_{k-l_1}^{s-i}\prod_{r=2}^{i+1}c_{k-l_r}\!\right)\! \!\left(\!\tilde\zeta_{l_1}^{(r)}\prod_{r=2}^{i+1}\xi_{l_r} - E\left(\!\tilde\zeta_{l_1}\prod_{r=2}^{i+1}\xi_{l_r}\!\right)\!\right)\!.
\end{eqnarray}
Replacing $l_j$ with $k-l_j$ for all $1\leq j\leq i+1$ in (\ref{betaterm0.5}), and taking
$$X_n\ =\ \sum_{k=1}^{n}\ \sum_{\substack{l_1,l_2,\hdots,l_{i+1}\\l_1\not\in\{l_2,\hdots,l_{i+1}\}}} \left(c_{k-l_1}^{s-i}\prod_{r=2}^{i+1}c_{k-l_r}\right) \left(\tilde\zeta_{l_1}^{(r)}\prod_{r=2}^{i+1}\xi_{l_r}\right)$$
in Lemma \ref{reducexpec}, with $z = \tau \in (1,2)$, we get that
\begin{eqnarray}\label{betaterm1}\nonumber
\!\!\!\!&&
E^\frac{1}{\tau}\!\!\left[\sup\limits_{n_r\le n< n_{r+1}}\left|\tilde T_n^{(r)}(i)\right|^\tau\right]
\\ \nonumber
\!\!\!\!&\stackrel{r}{\ll}&
E^\frac{1}{\tau}\!\!\left[\sup\limits_{n_r\le n< n_{r+1}}\!\left|\sum_{k=1}^{n}\ \sum_{\substack{l_1,l_2,\hdots,l_{i+1}\\l_1\not\in\{l_2,\hdots,l_{i+1}\}}} \left(c_{l_1}^{s-i}\prod_{r=2}^{i+1}c_{l_r}\right) \left(\tilde\zeta_{k-l_1}^{(r)}\prod_{r=2}^{i+1}\xi_{k-l_r}\right)\right|^\tau\  \right]
\\
\!\!\!\!&\leq&
E^\frac{1}{\tau}\!\!\left[\sup\limits_{n_r\le n< n_{r+1}}\!\left|\sum_{k=1}^{n}\ \sum_{l_1=-\infty}^{\infty} \left|c_{l_1}^{s-i}\tilde\zeta_{k-l_1}^{(r)}\right|\ \left|\sum_{l\in\mathbb{Z}\setminus\{l_1\}}c_l \xi_{k-l}\right|^i\ \right|^\tau\ \right].
\end{eqnarray}
\begin{eqnarray}\label{phidef}
\hspace{-4cm}\mbox{Define}\hspace{2.65cm}\phi_{k,q}\ =\ \left|\sum_{l\in \mathbb{R}\setminus\{q\}} c_l \xi_{k-l}\right|^i.
\end{eqnarray}
Noting that $\sum_{m\in\mathbb{Z}} |c_m^{s-i}| < \infty$ because $s-i \geq 2$, then using Jensen's inequality due to convexity of norms, we see that RHS of (\ref{betaterm1}) is upper bounded by
\begin{eqnarray}\nonumber\label{betatermhalf}
\!\!\!&&\ \ 
E^\frac{1}{\tau}\!\left[\ \left| \sum_{l_1=-\infty}^\infty\!\!\left|c_{l_1}^{s-i}\right|\sup\limits_{n_r\le n< n_{r+1}}\left(\sum_{k=1}^{n}\ \left|\tilde\zeta_{k-l_1}^{(r)}\right||\phi_{k,l_1}|\ \right)\right|^\tau\ \right]
\\ \nonumber
\!\!\!&=&\ 
\sum_{m=-\infty}^\infty\!|c_m^{s-i}|\ \
E^\frac{1}{\tau}\!\left[\ \left|\sum_{l_1 = -\infty}^{\infty}\frac{|c_{l_1}^{s-i}|}{\sum_{m} |c_m^{s-i}|}\sup\limits_{n_r\le n< n_{r+1}}\left(\sum_{k=1}^{n}\  \left|\tilde\zeta_{k-l_1}^{(r)}\right||\phi_{k,l_1}|\right)\right|^\tau\ \right]
\\
\!\!\!&\leq&\ 
\sum_{l_1=-\infty}^\infty\!|c_{l_1}^{s-i}|\ \ E^\frac{1}{\tau}\! \left[\sup\limits_{n_r\le n< n_{r+1}}\left|\sum_{k=1}^{n}\  \left|\tilde\zeta_{k-l_1}^{(r)}\right||\phi_{k,l_1}|\ \right|^\tau\ \right].
\end{eqnarray}
\textbf{Case 1: $i\geq 1$.} In this case, note that $\tau i<s$ follows since $i < \lfloor\frac{s-1}{2}\rfloor,\ \tau<2$. Then, by two applications of H\"older's inequality with $p_1 = \frac{s}{s-\tau i}\ $ and $\ p_2 = \frac{s}{\tau i}$ (both of which are positive, and their reciprocals sum to one), we get that the RHS of (\ref{betatermhalf}) is upper bounded by
\begin{eqnarray}\label{betaterm3}\nonumber
\!\!\!\!&&
\sum_{l_1=-\infty}^\infty\!|c_{l_1}^{s-i}|\  E^\frac{1}{\tau}\left[\sup\limits_{n_r\le n< n_{r+1}}\left|\sum_{k=1}^n \left|\tilde \zeta_{k-l_1}^{(r)}\right|^\frac{s}{s-\tau i} \right|^\frac{\tau(s-\tau i)}{s}\left|\sum_{j=1}^n \left|\phi_{j,l_1}\right|^\frac{s}{\tau i}\right|^\frac{\tau^2i}{s}\right]
\\ \nonumber
\!\!\!\!&\stackrel{r}{\ll}&\!
\sum_{l_1\in\mathbb{Z}}\!|c_{l_1}^{s-i}| E^\frac{s-\tau i}{s\tau}\!\!\left[\!\sup\limits_{n_r\le n< n_{r+1}}\!\left|\sum_{k=1}^n \left|\tilde \zeta_{k-l_1}^{(r)} \right|^\frac{s}{s-\tau i}\right|^\tau\right]\!E^\frac{i}{s}\!\!\left[\!\sup\limits_{n_r\le n< n_{r+1}}\! \left|\sum_{j=1}^n \left|\phi_{j,l_1}\right|^\frac{s}{\tau i}\right|^\tau \right]\!.
\end{eqnarray}
Since $\ \frac{s}{s-\tau i}\ $ and $\ \frac{s}{\tau i}\ $ are positive, we find that both $\ \sum_{k=1}^n \left|\tilde \zeta_{k-l_1}^{(r)} \right|^\frac{s}{s-\tau i}\ $ and $\ \sum_{j=1}^n \left|\phi_{j,l_1}\right|^\frac{s}{\tau i}\ $ are non-negative submartingales, which is shown in Shiryaev \cite[Page~475, Example~4]{Shiryaev}. Thus, using Doob's $L_p$ maximal inequality (see \cite[Page~493, Theorem~4]{Shiryaev}), and then Jensen's inequality (since $\tau > 1$), we get that
\begin{eqnarray}\nonumber\label{betaterm4}
\!\!\!\!\!\!&&
E^\frac{1}{\tau}\!\!\left[\sup\limits_{n_r\le n< n_{r+1}}\left|\tilde T_n^{(r)}(i)\right|^\tau\right]
\\ \nonumber
\!\!\!\!\!\!&\stackrel{r}{\ll}&
\sum_{l_1=-\infty}^\infty\!|c_{l_1}^{s-i}|\ E^\frac{s-\tau i}{s\tau}\left[ \left|\sum_{k=1}^{n_{r+1}-1}\left|\tilde \zeta_{k-l_1}^{(r)} \right|^\frac{s}{s-\tau i}\right|^\tau\right]E^\frac{i}{s}\left[ \left|\sum_{j=1}^{n_{r+1}-1} \left|\phi_{j,l_1}\right|^\frac{s}{\tau i}\right|^\tau\right]
\\ \nonumber
\!\!\!\!\!\!&\stackrel{r}{\ll}&
\sum_{l_1=-\infty}^\infty\!|c_{l_1}^{s-i}|\ 
E^\frac{s-\tau i}{s\tau}\!\left[\!(n_{r+1}-1)^{\tau-1}\!\sum_{k=1}^{n_{r+1}-1}\! \left|\tilde \zeta_{k-l_1}^{(r)} \right|^\frac{s\tau}{s-\tau i}\!\right]
\\
\!\!\!\!\!\!&&\hspace{3cm}\times
E^\frac{i}{s}\!\left[\!(n_{r+1}-1)^{\tau-1}\!\sum_{j=1}^{n_{r+1}-1} \!\left|\phi_{j,l_1}\right|^\frac{s}{i}\!\right].
\end{eqnarray}
Lemma \ref{dataexpec} directly implies that $\sup_{l_1\in\mathbb{Z}}\|\phi_{1,l_1}\|_{\frac{s}{i}}<\infty$. Since $s-i\geq 2$, $\{\tilde\zeta_{l}^{(r)}\}_{l\in\mathbb{Z}}$ are i.i.d., as are $\{\phi_{j,l_1}\}_{j\in\mathbb{N}}$, we get from (\ref{betaterm4}) that
\begin{eqnarray}\label{betatermend}\nonumber
&&\ 
E^\frac{1}{\tau}\!\left[\sup\limits_{n_r\le n< n_{r+1}}\left|\tilde T_n^{(r)}(i)\right|^\tau\right]
\\ \nonumber
&\stackrel{r}{\ll}&\ 
\sum_{l_1=-\infty}^\infty\!|c_{l_1}^{s-i}|\ 
E^\frac{s-\tau i}{s\tau}\left[(n_{r+1}-1)^{\tau} \left|\tilde \zeta_{1}^{(r)} \right|^\frac{s\tau}{s-\tau i}\right]
E^\frac{i}{s}\left[(n_{r+1}-1)^{\tau}\left|\phi_{1,l_1}\right|^\frac{s}{i}\right]
\\
&\stackrel{r}{\ll}&\ 
\sum_{l_1=-\infty}^\infty\!|c_{l_1}^{s-i}|\ n_r\left\|\tilde \zeta_{1}^{(r)}\right\|_{\frac{s\tau}{s-\tau i}}\|\phi_{1,l_1}\|_{\frac{s}{i}}
\ \ \stackrel{r}{\ll}\ \
n_r\|\tilde \zeta_{1}^{(r)}\|_{\frac{s\tau}{s-\tau i}}\ .
\end{eqnarray}
\textbf{Case 2: $i = 0$}. In this case, we get that $|\phi_{k,l_1}|=1$, and from (\ref{betatermhalf}), we get that
$$E^\frac{1}{\tau}\!\!\left[\sup\limits_{n_r\le n< n_{r+1}}\left|\tilde T_n^{(r)}(i)\right|^\tau\right]\ \stackrel{r}{\ll}\ \sum_{l_1=-\infty}^\infty\!|c_{l_1}^{s}|\ \ E^\frac{1}{\tau}\! \left[\sup\limits_{n_r\le n< n_{r+1}}\left|\sum_{k=1}^{n}\  \left|\tilde\zeta_{k-l_1}^{(r)}\right|\ \right|^\tau\ \right].$$
Again, using Doob's $L_p$ maximal inequality, Jensen's inequality, the fact that $\ \sum_{k=1}^n \left|\tilde \zeta_{k-l_1}^{(r)} \right|$ is a non-negative submartingale, and that $\{\tilde\zeta_{l}^{(r)}\}_{l\in\mathbb{Z}}$ are i.i.d., we proceed as in (\ref{betatermend}) to get
\begin{eqnarray}\nonumber\label{betatermextra}
E^\frac{1}{\tau}\left[\sup\limits_{n_r\le n< n_{r+1}}\left|\tilde T_n^{(r)}(i)\right|^\tau\right]
&\stackrel{r}{\ll}&
\sum_{l_1=-\infty}^\infty|c_{l_1}^{s}|\ 
E^\frac{1}{\tau}\!\!\left[(n_{r+1}-1)^{\tau-1}\sum_{k=1}^{n_{r+1}-1} \left|\tilde \zeta_{k-l_1}^{(r)} \right|^{\tau}\right]
\\ 
&\stackrel{r}{\ll}&
\ n_r\|\tilde \zeta_{1}^{(r)}\|_{\tau}\ .
\end{eqnarray}
Thus, for all $\ i\in \left\{0,1,\hdots,\left\lfloor\frac{s-1}{2}\right\rfloor\right\}$, we get from (\ref{betatermend}) and (\ref{betatermextra}), that
\begin{eqnarray}\label{betaterm}
E^\frac{1}{\tau}\left[\sup\limits_{n_r\le n< n_{r+1}}\left|\tilde T_n^{(r)}(i)\right|^\tau\right]\ \stackrel{r}{\ll}\ n_r\|\tilde \zeta_{1}^{(r)}\|_{\frac{s\tau}{s-\tau i}}\ .
\end{eqnarray}
Now, we choose $\tau>1$ small enough so that $\alpha_i > \frac{s\tau}{s-\tau i}$, which is possible since $\alpha_i = \frac{s}{s-i}\alpha_0 > \frac{s}{s-i}$, and $\frac{s\tau}{s-\tau i}$ is continuous and increasing for $\tau\in(1,\alpha_i)$. Hence by (\ref{momboundtilde}) with $z = \frac{s\tau}{s-\tau i}$ and (\ref{betaterm}), there exists $\mathcal{T}_i\in(1,\alpha_i)\ $ such that $\ \forall\ \tau\in(1,\mathcal{T}_i)$,
\begin{eqnarray}\label{remmoment}
E\left[\sup\limits_{n_r\le n< n_{r+1}}
\left|\tilde T_n^{(r)}(i)\right|^\tau \right]\ \ \stackrel{r}{\ll}\ \ 
n_r^{\tau-\frac{\kappa(\alpha_i-\frac{s\tau}{s-\tau i})}{\frac{s}{s-\tau i}(2-\alpha_i)}}.
\end{eqnarray}

\subsection{Final Rate of Convergence for Theorem \ref{mostgen}}
\noindent
Finally, we shall use the Borel-Cantelli Lemma to combine the results of the last two sections and prove Theorem \ref{mostgen}. Notice that in $\sum_{k=1}^n (d_k-d)$ (from Theorem \ref{mostgen}), the light-tailed terms are $S_n(q,\lambda_q)$ (from (\ref{sqlambdaq})) over all partitions where $a_1\leq \frac{s}{2}$, since their second moments are finite. The heavy-tailed terms are $S_{n}^\star(i)$ (from (\ref{sstar})) over $i\in\left\{0,1,\hdots,\left\lfloor\frac{s-1}{2}\right\rfloor\right\}$. We thus have
\begin{eqnarray}\label{final}
\sum_{k=1}^n (d_k-d)\ =\! \sum_{\substack{\lambda_q=(a_1,\hdots,a_q)\\a_1\leq\frac{s}{2}}}S_n(q,\lambda_q)\ +\! \sum_{i\in\left\{0,1,\hdots,\left\lfloor\frac{s-1}{2}\right\rfloor\right\}}S_{n}^\star(i)\ .\qquad
\end{eqnarray}
$\bullet\ $ First, we handle the light-tailed terms. In Lemma \ref{psilemma}, taking $\psi_{l_r}^{(r)} = \xi_{l_r}^{a_r}\ $ for $\ 1\leq r\leq q$, and$\ \delta = 1$, we see that $Y_{n',n,\delta}^{\lambda_q}$ becomes $S_n(q, \lambda_q)-S_{n'}(q, \lambda_q)$, and we get the same results as in (\ref{res}) and (\ref{res1}). Thus, proceeding along the lines of (\ref{max0}\ -\ \ref{sconv3}), we get that
\begin{eqnarray}\label{lt}
\lim_{n \rightarrow \infty}{n^{-\frac{1}p}}{S_n(q,\lambda_q)}\ =\ 0 \qquad\qquad \mbox{a.s.}
\end{eqnarray}
for the values of $p$ as mentioned in (\ref{longonlyp},\ref{longonlyp2}), in the statement of Theorem \ref{longonlygen}.\\
$\bullet\ $ Now we deal with the heavy-tailed terms. We fix $i\in\left\{0,1,\hdots,\left\lfloor\frac{s-1}{2}\right\rfloor\right\}$, which fixes $S_{n}^\star(i)$, and due to (\ref{offdiagexpres}), consider $T_n(i)$ instead of $S_{n}^\star(i)$. First, we consider the case where $s>2$. From (\ref{maxtrunc1.5},\ref{remmoment}), Markov's Inequality, and the fact that $l_{1,\sigma}(n_r)=n_r^{2-2\sigma}$ (since $\sigma<1$), we get that, there exists $\mathcal{T}_i$ such that $\ \forall\ 1<\tau<\mathcal{T}_i$, 
\begin{eqnarray}\label{Pmaxzeta}\nonumber
&&\ P\left(\sup\limits_{n_r\le n< n_{r+1}}
\left|T_n(i)\right|>2\epsilon n_r^\frac{1}p\right)
\\ \nonumber
&\le &\
\frac{1}{\epsilon^2 n_r^\frac2p}E\left[\sup\limits_{n_{r}\le n<n_{r+1}}\left|\overline T_n^{(r)}(i)\right|^{2}\right]
\ \ +\ \
\frac{1}{\epsilon^\tau n_r^\frac{\tau}{p}}E\left[\sup\limits_{n_{r}\le n<n_{r+1}}\left|\tilde T_n^{(r)}(i)\right|^\tau\right]
\\ \nonumber
&\stackrel{r}{\ll}&\ 
r^2 \left[\left(n_r^{1-\frac2p}l_{1,\sigma}(n_r)\right)\vee\left( n_r^{1+\kappa-\frac2p}\right)\right]\ \ +\ \ n_r^{\tau-\frac{\kappa(\alpha_i-\frac{s\tau}{s-\tau i})}{\frac{s}{s-\tau i}(2-\alpha_i)}-\frac{\tau}{p}}
\\
&\stackrel{r}{\ll}&\ 
r^2 \left[\left(n_r^{3-2\sigma-\frac2p}\right)\vee\left(n_r^{1-\frac{\alpha_i}{p}}\right)\right]\ \ +\ \ n_r^{\tau-\frac{\alpha_i(s-\tau i)}{ps}},
\end{eqnarray}
by letting $\kappa = \frac{2-\alpha_i}p$. Note that $(3-2\sigma-\frac2p)\vee (1-\frac{\alpha_i}{p}) < 0\ $ implies that $\ p < \alpha_i \wedge \frac{2}{3-2\sigma}$. Next, note that $\tau - \frac{(s-\tau i)\alpha_i}{ps} < 0\ $ if and only if $\ p < \alpha_i \left(\frac{s-\tau i}{s\tau}\right)$. But for any $p < \alpha_0 = \alpha_i\left(\frac{s-i}s\right)$, we select $\tau>1$ small enough such that $p < \alpha_i\left(\frac{s-\tau i}{s\tau}\right)$. Hence, from (\ref{Pmaxzeta}), we get that $\sum_{r=1}^\infty P\left(\sup\limits_{n_r\le n< n_{r+1}}
\left|T_n(i)\right| > 2\epsilon n_r^\frac{1}p\right) <\infty,$ for
\begin{eqnarray}\label{Psumzeta}
p < \alpha_0 \wedge \frac{2}{3-2\sigma}.
\end{eqnarray}
$\bullet\ $ When $s=2$, using (\ref{maxtrunc2}, \ref{remmoment}), proceeding along the lines of (\ref{Pmaxzeta}, \ref{Psumzeta}), we get that $\sum_{r=1}^\infty P\left(\sup\limits_{n_r\le n< n_{r+1}}
\left|T_n(i)\right| > 2\epsilon n_r^\frac{1}p\right) <\infty,$ for
\begin{eqnarray}\label{Psumzeta1}
p < \alpha_0.
\end{eqnarray}
$\bullet\ $ Lastly, when s is even, and $\xi_1$ is symmetric, using (\ref{maxtrunc3}, \ref{remmoment}), and again proceeding along the lines of (\ref{Pmaxzeta}, \ref{Psumzeta}), we get that\\$\sum_{r} P\left(\sup\limits_{n_r\le n< n_{r+1}}
\left|T_n(i)\right| > 2\epsilon n_r^\frac{1}p\right) <\infty,$ for
\begin{eqnarray}\label{Psumzeta2}
p < 2\wedge\alpha_0\wedge\frac{1}{2-2\sigma}.
\end{eqnarray}
Hence, for the values of $p$ in (\ref{Psumzeta}, \ref{Psumzeta1}, \ref{Psumzeta2}), from the Borel-Cantelli Lemma, we get that
\begin{eqnarray}\label{ht}
\lim_{n \rightarrow \infty}{n^{-\frac{1}p}}{T_n(i)} = 0\quad\mbox{a.s.},\qquad \mbox{and hence}\quad 
\lim_{n \rightarrow \infty}{n^{-\frac{1}p}}{S_{n}^\star(i)} = 0\quad\mbox{a.s.},
\end{eqnarray}
due to (\ref{sstar}, \ref{offdiagexpres}). From (\ref{lt}, \ref{ht}) and Remark \ref{sigmasless1}, we get that
$$\lim\limits_{n\rightarrow\infty}{n^{-\frac{1}{p}}}\sum_{k=1}^n (d_k-d)=0\qquad\mbox{a.s.}\ ,$$
for the values of $p$ in the statement of Theorem \ref{mostgen}. 
This proves Theorem \ref{mostgen}. \hfill$\Box$

\begin{remark}\label{difffromgen}
     Here we underline the notational changes that would have to be made to prove the case where all the innovations and the LRD coefficients are allowed to be unequal (see Remark \ref{sigmasless1}). In (\ref{sqlambdaq}), our decomposition will require partitions of $\{1,2,\hdots,s\}$ instead of $s$. Recalling (\ref{linprocess}) in General $\mathbb{R}$-valued product case, we define $S_n(q,\lambda_q)$ as
    \begin{eqnarray}\nonumber
	   \sum_{k=1}^{n}\ \sum_{l_1\neq l_2\neq \hdots \neq l_q} \left(\prod_{r=1}^{q}\prod_{w\in A_r} c_{k-l_r}^{(w)}\right) \left(\prod_{r=1}^{q}\prod_{w\in A_r}\xi_{l_r}^{(w)} - E\left(\prod_{r=1}^{q}\prod_{w\in A_r}\xi_{l_r}^{(w)}\right)\right),
	\end{eqnarray}
	where $q$ ranges over $\{1,2,\hdots,s\}$, and $\lambda_q = (A_1, A_2, \hdots, A_q)$ is a decreasing partition of the set $\{1,2,\hdots,s\}$, i.e. it satisfies $\bigcup_{r=1}^q A_r = \{1,2,\hdots,s\}$, $\sum_{r=1}^q |A_r| = s$, and $|A_1| \geq \hdots \geq |A_q|\geq 1$. For the entirety of the proof,  $\prod_{w\in A_r} c_{k-l_r}^{(w)}$ and $\prod_{w\in A_r}\xi_{l_r}^{(w)}$ act as proxies for $c_{k-l_r}^{a_r}$ and $\xi_{l_r}^{a_r}$ respectively, but because of (Tail) and (Decay), the steps remain the same. The proofs of Lemmas \ref{covlemma}, \ref{psilemma}, and the heavy-tailed portion also go through with notational changes, the one exception being that instead of using Lemma \ref{inequality} as stated, we use the slightly modified bound stated here
	$$\sum_{l\in\mathbb{R}\setminus\{j,k\}}|j-l|^{-\gamma_1} |k-l|^{-\gamma_2}\ \stackrel{j,k}{\ll}\ |j-k|^{1-\gamma_1-\gamma_2},\quad \text{where}\ \gamma_1,\gamma_2 \in \left(\frac12,1\right).$$
	That makes one of the expressions in the bound for $p$ change from $2\sigma$ to $\min_{1\leq i\leq j\leq s}\{\sigma_i+\sigma_j\}$, in the statement of Theorem \ref{mostgen}.
\end{remark}

\appendix

\section{Technical Lemmas}\label{apptech}
\noindent
The following simple lemmas are used in some of the proofs of our paper. The proofs of Lemmas \ref{inequality} and \ref{newineq} are provided in the supplementary materials. Please recall that notation like $\stackrel{j,k}{\ll}$ is explained in our notation list in Section \ref{list}.
\begin{lemma}\label{inequality}
	For $j, k\in\mathbb Z$, $j \neq k$ and $\gamma > \frac12$, we have,
	$$\sum_{\substack{l=-\infty\\l\not\in \{j,k\}}}^{\infty} |j-l|^{-\gamma} |k-l|^{-\gamma} \quad \stackrel{j,k}{\ll} \quad
	\left\{ \begin{array}{ll} |j-k|^{1-2\gamma}, & \gamma\in \left(\frac12,1\right)\\ |j-k|^{-1} \log(|j-k|+1),\quad & \gamma = 1\\|j-k|^{-\gamma},& \gamma >1\end{array}\right..$$
\end{lemma}
The following lemma now follows directly by Lemma \ref{inequality}.
\begin{lemma}\label{newineq}
	For $j, k\in\mathbb Z$, $j \neq k$ and $\gamma \in \left(\frac12,1\right)$, we have,
	$$\sum_{\substack{l=-\infty\\l\not\in \{j,k\}}}^{\infty} |j-l|^{-\gamma} |k-l|^{-2\gamma} \qquad \stackrel{j,k}{\ll} \qquad |j-k|^{-\gamma}\ .$$
\end{lemma}

The following lemma follows easily by Triangle Inequality, Minkowski’s Inequality and Jensen’s Inequality.
\begin{lemma}\label{reducexpec}
	Let $z > 1,\ n_r = 2^r\ \ \forall\ r\in \mathbb{N}$, and $\{X_n\}_{n\in \mathbb{N}}$ be random variables such that $E\big[|X_n|^z\big]\ <\ \infty$. Then, we have
	$$E^\frac{1}{z}\left[\sup\limits_{n_r\le n< n_{r+1}}\left|X_n - E\left(X_n\right)\right|^z\right]\ \ \stackrel{r}{\ll}\ \ E^\frac{1}{z}\left[\left|\sup\limits_{n_r\le n< n_{r+1}}\left|X_n\right|\ \right|^z\right]\ .$$
\end{lemma}

The following lemma guarantees the existence of the $s$th moment of a two-sided LRD linear process as long as the $s$th moment of its innovations are finite.
It follows from Samorodnitsky \cite[Theorem~1.4.1]{Samorod} and triangle inequality.
\begin{lemma}\label{dataexpec}
	Let $s\in \mathbb{N}$ and $\left\{\xi_{l}\right\}_{l\in\mathbb Z}$ be i.i.d.\ zero-mean random variables such that $E\left[|\xi_1|^{s\vee 2} \right]<\infty$, and $\{c_l\}_{l\in\mathbb Z}$ satisfy $\ \sup\limits_{l\in\mathbb Z}|l|^\sigma|c_l|<\infty$, for some $\sigma \in \left(\frac{1}{2},1\right)$. Let $1<i<s$, and $\ \phi_{k,q}=\left|\sum_{l\in \mathbb{R}\setminus\{q\}} c_l \xi_{k-l}\right|^i$. Then, $\sup_{q\in\mathbb{Z}}\|\phi_{k,q}\|_{\frac{s}{i}}<\infty$.
\end{lemma}

\section{Classical Theorems}\label{appclassic}
\noindent
Loeve \cite[Section~17, Theorem~A, case~4]{Loeve} provides the following statement of the Marcinkiewicz-Zygmund strong law of large numbers.
\begin{theorem}[Marcinkiewicz-Zygmund Strong Law of Large Numbers]\label{msllndef}
	Let $\ \{X_n\}_{n\in\mathbb{Z}}$ be a sequence of i.i.d. random variables, and let $\ 0<p<2$. Then, $E\left[|X_1|^p\right]<\infty\ $ if and only if
	\begin{eqnarray}\nonumber
	\lim_{n\rightarrow\infty} n^{-\frac{1}{p}}\sum_{k=1}^{n}(X_k-c)=0\quad\mbox{a.s.\ ,\quad where}\ \ c=\left\{ \begin{array}{ll} 0, &\ p<1\\ E(X_1), &\ p\geq 1 \end{array}\right..
	\end{eqnarray}
\end{theorem}
More generally, for a stationary time series $\{X_n\}_{n\in\mathbb{Z}}$ with given conditions on $\{X_n\}$, any result regarding the almost sure convergence of $n^{-\frac1p}$$\sum_{k=1}^{n}(X_n-c)$ for some constant $c$ and some $p\in (0,2)$, is known as a Marcinkiewicz-Zygmund strong law, or simply a Marcinkiewicz strong law of order $p$.

Lastly, we present the following Theorem, which follows from a theorem developed by Serfling (see Stout \cite[Theorem~2.4.1]{Stout}). The full derivation is provided in the supplementary document.  
\begin{theorem}\label{stout}
	Let $\{Z_k\}_{k\in\mathbb{N}}$ be a time series with finite second moments, and $f$ be a super-additive function on $\mathbb{N}$, such that
	\begin{eqnarray}\nonumber
	E\left[\left(\sum_{i=n'+1}^{n} Z_i \right)^2\right]\ &\leq&\ f(n-{n'})\qquad\!\forall\ n'<n\in\mathbb{N}\cup\{0\}\ .
	\end{eqnarray}
	Then, for $n_{r}=2^{r},\ r\in\mathbb{N}\cup\{0\}$, and $n',n\in\mathbb{N}$, we have
	\begin{eqnarray}\nonumber
	E\left[\max_{n_r\leq n'<n<n_{r+1}}\left(\sum_{i=n'+1}^{n} Z_i\right)^2 \right]\ \ \stackrel{r}\ll\ \ r^2 f(n_r)\ .
	\end{eqnarray}
\end{theorem}

\section{Supplementary Document}\label{supplemdoc}
\noindent
Lemmas 2, 3 and 4, as well as Theorem 5 of the paper are restated and proved in this supplementary document. Definitions, equations and references from the paper are often referred to in the proofs.

\begin{manuallemma}{2}\label{psilemma}
	Let $n'<n\in \mathbb{N}\cup\{0\}$, $s \in \mathbb{N},\ \delta\geq 1,\ \lambda_q = (a_1, a_2, \hdots, a_q)$ is a decreasing partition of $s$, and $v = \left|\{1\leq r\leq q: a_r = 1\}\right|$. Let $\{c_l\}_{l\in\mathbb Z}$ satisfy $\ \sup\limits_{l\in\mathbb Z}|l|^\sigma|c_l|<\infty,\ \ \mbox{for some}\ \ \sigma \in \left(\frac{1}{2},1\right)$, and $\{(\psi_l^{(1)},\hdots,\psi_l^{(q)})\}_{l\in \mathbb{Z}}$ be i.i.d $\ \mathbb{R}^q$-valued random vectors, such that
	\begin{eqnarray}\label{psiassump}
	\quad\left\{ \begin{array}{l} E\left(\psi_1^{(r)}\right)\qquad \ll\  \textbf{1}_{\{1\leq r\leq q-v\}},
	\\
	E\left[\left(\psi_1^{(r)}\right)^2\right]\ \ \ll\ \delta\textbf{1}_{\{r=1\}}+\textbf{1}_{\{r\neq 1\}},
	\end{array}\right.\qquad \forall\ \ 1\leq r \leq q\ .
	\end{eqnarray}
	\begin{eqnarray}\nonumber
	\hspace{-0.5cm}\mbox{Define},\ Y_{n',n,\delta}^{\lambda_q} = \sum_{k=n'+1}^{n}\  \sum_{\boldsymbol\ell\in\Delta}\left(\prod_{r=1}^{q} c_{k-l_r}^{a_r}\right) \left(\prod_{r=1}^{q}\psi_{l_r}^{(r)} - E\left(\prod_{r=1}^{q}\psi_{l_r}^{(r)}\right)\right).
	\end{eqnarray}
	Then,$\ E\left[(Y_{n',n,\delta}^{\lambda_q})^2\right]\!\stackrel{n',n,\delta}{\ll}
	\!\left\{ \begin{array}{ll}\delta\ (n-n'), &a_q \geq 2,\\ \delta\ (n-n')\ l_{s,\sigma}(n-n'), &a_1 = 1,\\ (\delta\ (n-n'))\vee((n-n')\ l_{1,\sigma}(n-n')),\! &a_q=1,a_1\geq 2,\end{array}\right.$\\\\
	where $\boldsymbol\ell$ and $l_{s,\sigma}$ are defined in the Notation List in Subsection 1.2. Further, if s is even and $E\left(\psi_1^{(r)}\right) = 0\ $for odd $a_r$, then this bound can be tightened to
	$$E\left[(Y_{n',n,\delta}^{\lambda_q})^2\right]\ \ \stackrel{n',n,\delta}{\ll}\ \ (\delta\ (n-n'))\ \vee\ ((n-n')\ l_{2,\sigma}(n-n')),$$
	when $a_q = 1$ and $a_1 \geq 2$.
\end{manuallemma}
\begin{proof}

We first bound the second moment of $Y_{n',n,\delta}^{\lambda_q}$.
\begin{eqnarray}\label{main1}\nonumber
\!\!\!\!\!\!\!\!\!\!E\left[(Y_{n',n,\delta}^{\lambda_q})^2\right]\!\!
& = &\!\!\!\!
\sum\limits _{k=n'+1}^{n}\sum\limits _{j=n'+1}^{n}\ \sum_{l_1\neq l_2\neq \hdots\neq l_q}\ \sum_{m_1\neq m_2\neq \hdots\neq m_q}\!\!\left(\prod_{r=1}^{q} c_{j-m_r}^{a_r} c_{k-l_r}^{a_r}\!\!\right)
\\\nonumber
\!\!\!\!\!\!\!\!\!\!\!\!\!\!\!\!&&\hspace{0.6cm} \left[E\left(\prod_{r=1}^{q}\psi_{l_r}^{(r)}\psi_{m_r}^{(r)}\right) - E\left(\prod_{r=1}^{q}\psi_{l_r}^{(r)}\right)E\left(\prod_{r=1}^{q}\psi_{m_r}^{(r)}\right)\right]
\\ \nonumber
\!\!\!\!\!\!\!\!\!\!\!\!\!\!\!\!&\leq&
\sum_{k=n'+1}^{n}\sum_{j=n'+1}^{n} \sum_{(\boldsymbol\ell,\textit{\textbf{m}})\in\Delta\times \Delta}\left(\prod_{r=1}^{q} \left|c_{j-m_r}^{a_r}\right| \left|c_{k-l_r}^{a_r}\right|\right)
\\
&&\hspace{0.1cm} \left|E\left(\prod_{r=1}^{q}\psi_{l_r}^{(r)}\psi_{m_r}^{(r)}\right) - E\left(\prod_{r=1}^{q}\psi_{l_r}^{(r)}\right)E\left(\prod_{r=1}^{q}\psi_{m_r}^{(r)}\right)\right|.
\end{eqnarray}
Please refer to Definition 5 for the notation from here on. Observe that the summation in (\ref{main1}) is over $\Delta\times \Delta$. Based on $q$ and $v = \left|\{1\leq r\leq q: a_r = 1\}\right|$, we can partition $\Delta\times \Delta$ into the sets $\Delta\times \Delta(V_1,...,V_6,\nu)$.
For sets $V_1,\hdots,V_6$ and matching function $\nu$ as in Definition 4, define
\begin{eqnarray}\nonumber\label{Sdef}
\!\!\!\!\!\!\!\!\!\!\!\!\!\!\!S(V_1,\hdots,V_6,\nu)\!\!\!
&=&\!\!\!
\sum\limits_{k=n'+1}^{n}\sum\limits _{j=n'+1}^{n} \sum_{(\boldsymbol\ell,\textit{\textbf{m}})\in\Delta\times \Delta(V_1,...,V_6,\nu)}\!\!\left(\prod_{r=1}^{q} \left|c_{j-m_r}^{a_r}\right|\left|c_{k-l_r}^{a_r}\right|\!\right)
\\
\!\!\!\!\!\!\!\!\!\!\!\!\!\!\!\!\!&&\!\hspace{0.3cm}
\left|E\!\left(\prod_{r=1}^{q}\psi_{l_r}^{(r)}\psi_{m_r}^{(r)}\!\right) - E\!\left(\prod_{r=1}^{q}\psi_{l_r}^{(r)}\!\right)E\!\left(\prod_{r=1}^{q}\psi_{m_r}^{(r)}\!\right)\right|\!,
\end{eqnarray}
where $\boldsymbol\ell = (l_1,\hdots,l_q)$ and $\textit{\textbf{m}} = (m_1,\hdots,m_q)$ are as in the Notation List in Subsection 1.2. Using the fact that for a given $q$, there can only be a finite number of possibilities for $V_1,\hdots,V_6$ and $\nu$, we get from (\ref{main1}) and (\ref{Sdef}), that
\begin{eqnarray}\label{YS}
E\left[(Y_{n',n,\delta}^{\lambda_q})^2\right]\ \  \stackrel{n',n,\delta}{\ll}\ \max_{V_1,\hdots,V_6,\ \nu} S(V_1,\hdots,V_6,\nu)\ .
\end{eqnarray}
Observe that when $\left|V_1\right|>0\ $ or $\ \left|W_1\right|>0$, $S(V_1,\hdots,V_6,\nu)=0$ according to Lemma 1, and need not be considered in (\ref{YS}). Hence we assume that $\left|V_1\right|=\left|W_1\right|=0$. From Remark 7, recall that $\ \left|V_1\right|+\left|V_2\right|+\left|V_3\right|=\left|W_1\right|+\left|\nu(V_2)\right|+\left|\nu(V_4)\right|=v$. Due to injectivity of $\nu$, we have $|\nu(V_r)|=|V_r|$ for $2\leq r\leq 5$, so when $\ \left|V_1\right|=\left|W_1\right|=0$, we get our second observation, i.e. $\left|V_3\right|=\left|V_4\right|$. Similarly, since $\left|V_1\right|+\hdots+\left|V_6\right|=\left|W_1\right|+\left|\nu(V_2)\right|+\hdots+\left|\nu(V_5)\right|+\left|W_6\right|=q$, using $\ \left|V_1\right|=\left|W_1\right|=0$, we get that $\left|V_6\right|=\left|W_6\right|$. 
Hence, we only need to consider those terms $S(V_1,\hdots,V_6,\nu)$, where
\begin{eqnarray}\label{vassump}
\quad\left\{ \begin{array}{l} \left|V_1\right|=\left|W_1\right|=0,
\\
\left|V_3\right|=\left|V_4\right|,
\\
\left|V_6\right|=\left|W_6\right|.
\end{array}\right.
\end{eqnarray}
We now fix sets $V_1,\hdots,V_6$ and matching function $\nu$, from Definition 5, satisfying (\ref{vassump}). To find an upper bound of $S(V_1,\hdots,V_6,\nu)$, we use Lemma 1 and define
\begin{eqnarray}\label{rhobound}
\rho_{u_2,\hdots,u_6}\ =\ 
\left\{ \begin{array}{ll} 1,\quad &0<u_6<q,\ u_4=u_5=0\\ \delta,\quad &\mbox{otherwise}\end{array}\right..
\end{eqnarray}
Using (\ref{Sdef},\ref{vassump},\ref{rhobound}), and Lemma 1, we group coefficients according to $V_1,\hdots,V_6$, and $\nu$ to get that
\begin{eqnarray}\nonumber
\!\!\!\!\!\!\!\!\!\!\!\!\!\!\!&&
S(V_1,\hdots,V_6,\nu)
\\ \nonumber
\!\!\!\!\!\!\!\!\!\!\!\!\!\!\!&\stackrel{n',n,\delta}{\ll} &
\sum\limits_{k=n'+1}^{n}\sum\limits _{j=n'+1}^{n}\ \sum_{(\boldsymbol\ell,\textit{\textbf{m}})\in\Delta\times \Delta(V_1,...,V_6,\nu)}\left(\prod_{r=1}^{q} \left|c_{j-m_r}^{a_r}\right| \left|c_{k-l_r}^{a_r}\right|\right)\rho_{\left|V_2\right|,\hdots,\left|V_6\right|}
\end{eqnarray}
\begin{eqnarray}\nonumber\label{mainhalf}
\!\!\!\!\!\!\!\!\!\!\!\!\!\!\!&\stackrel{n',n,\delta}{\ll} &
\rho_{\left|V_2\right|,\hdots,\left|V_6\right|}\ \sum\limits_{k=n'+1}^{n}\sum\limits _{j=n'+1}^{n}\ \sum_{(\boldsymbol\ell,\textit{\textbf{m}})\in\Delta\times \Delta(V_1,...,V_6,\nu)}\left(\prod_{r\in W_6} |c_{j-m_r}^{a_r}|\right)
\\ \nonumber
\!\!\!\!\!\!\!\!\!\!\!\!\!\!\!&&\!\!\!\!
\left(\prod_{r\in V_6}|c_{k-l_r}^{a_r}| \right)\!\left(\prod_{r\in V_5}|c_{j-m_{\nu(r)}}^{a_{\nu(r)}}||c_{k-l_r}^{a_r}|\right)\!\left(\prod_{r\in V_4}|c_{j-m_{\nu(r)}}^{a_{\nu(r)}}||c_{k-l_r}^{a_r}|\right)
\\
\!\!\!\!\!\!\!\!\!\!\!\!\!\!\!&&\hspace{2cm}
\left(\prod_{r\in V_3}|c_{j-m_{\nu(r)}}^{a_{\nu(r)}}||c_{k-l_r}^{a_r}| \right)\!\left(\prod_{r\in V_2}|c_{j-m_{\nu(r)}}^{a_{\nu(r)}}||c_{k-l_r}^{a_r}| \right)\!.
\end{eqnarray}
Note that $\ a_r\geq 2\ $ (hence $\ c_l^{a_r}\leq c_l^2\ $) for $\ r\in V_4\cup V_5\cup V_6\cup W_6$, and $a_r=1\ $ for $r\in V_2\cup V_3$. Next, for $r\in V_2\cup V_3\cup V_4\cup V_5$, we note that $l_r=m_{\nu(r)}$ in (\ref{mainhalf}), then bring in the summations and extend them over all integers, to get
\begin{eqnarray}\nonumber\label{main2}
\!\!\!\!\!\!\!\!\!\!\!\!&&
S(V_1,\hdots,V_6,\nu)
\\ \nonumber
\!\!\!\!\!\!\!\!\!\!\!\!&\stackrel{n',n,\delta}{\ll} &
\rho_{\left|V_2\right|,\hdots,\left|V_6\right|}\ \sum\limits_{k=n'+1}^{n}\sum\limits _{j=n'+1}^{n}\left(\prod_{r\in W_6}\sum_{m_r=-\infty}^{\infty} |c_{j-m_r}^2|\right)\left(\prod_{r\in V_6}\sum_{l_r = -\infty}^{\infty}|c_{k-l_r}^2| \right)
\\ \nonumber
&&\hspace{2.2cm}
\left(\prod_{r\in V_5}\sum_{l_r = -\infty}^{\infty} |c_{j-l_r}^2||c_{k-l_r}^2| \right)
\left(\prod_{r\in V_4}\sum_{l_r = -\infty}^{\infty} |c_{j-l_r}||c_{k-l_r}^2| \right)
\\ \nonumber
&&\hspace{2.3cm}
\left(\prod_{r\in V_3}\sum_{l_r = -\infty}^{\infty} |c_{j-{l_r}}^2||c_{k-l_r}| \right)\left(\prod_{r\in V_2}\sum_{l_r = -\infty}^{\infty} |c_{j-l_r}||c_{k-l_r}|\right)
\\ \nonumber
\!\!\!\!\!\!\!\!\!\!\!\!&\stackrel{n',n,\delta}{\ll} &	\rho_{\left|V_2\right|,\hdots,\left|V_6\right|}\ \sum\limits_{k=n'+1}^{n}\sum\limits _{j=n'+1}^{n}\left(\sum_{m = -\infty}^{\infty}\!\! |c_{j-m}^2|\!\right)^{\left|W_6\right|}\left(\sum_{l = -\infty}^{\infty}\!\! |c_{j-l}^2|\right)^{\left|V_6\right|}
\\ \nonumber
&&\hspace{2.3cm} 
\left(\sum_{l = -\infty}^{\infty}\!\!|c_{j-l}^2||c_{k-l}^2|\right)^{\left|V_5\right|}\left(\sum_{l = -\infty}^{\infty}|c_{j-l}||c_{k-l}^2| \right)^{\left|V_4\right|}
\\
&&\hspace{2.4cm}
\left(\sum_{l = -\infty}^{\infty}|c_{j-l}^2||c_{k-l}| \right)^{\left|V_3\right|}\left(\sum_{l = -\infty}^{\infty} |c_{j-l}||c_{k-l}|\right)^{\left|V_2\right|}\!\!\!\!.
\end{eqnarray}\\
Applying Lemma \ref{inequality} with $\gamma = \sigma,\ 2\sigma$ and Lemma \ref{newineq} with $\gamma = \sigma$, we have
\begin{eqnarray}\label{both2}\nonumber
\sum_{l = -\infty}^{\infty}|c_{j-l}^2||c_{k-l}^2| &\stackrel{n',n,\delta}{\ll}&
\left\{ \begin{array}{ll} 1 + \sum_{\substack{l=-\infty\\l\neq j}}^{\infty}|j-l|^{-4\sigma},\  &j=k\\ \sum_{\substack{l=-\infty\\l\not\in \{j,k\}}}^{\infty} |j-l|^{-2\sigma} |k-l|^{-2\sigma}+|j-k|^{-2\sigma},\ &j\neq k\end{array}\right.\\
&\stackrel{n',n,\delta}{\ll}&
\left\{ \begin{array}{ll} 1,\quad &j=k\\ |j-k|^{-2\sigma},\quad &j\neq k\end{array}\right.
\end{eqnarray}
\begin{eqnarray}\label{one2one1}\nonumber
\sum_{l = -\infty}^{\infty}|c_{j-l}||c_{k-l}^2| &\stackrel{n',n,\delta}{\ll}&
\left\{ \begin{array}{ll} 1 + \sum_{\substack{l=-\infty\\l\neq j}}^{\infty}|j-l|^{-3\sigma},\quad &j=k\\ \sum_{\substack{l=-\infty\\l\not\in \{j,k\}}}^{\infty} |j-l|^{-\sigma} |k-l|^{-2\sigma}+|j-k|^{-\sigma},\quad &j\neq k\end{array}\right.\\
&\stackrel{n',n,\delta}{\ll}&
\left\{ \begin{array}{ll} 1,\quad &j=k\\ |j-k|^{-\sigma},\quad &j\neq k\end{array}\right.
\end{eqnarray}
\begin{eqnarray}\label{both1}\nonumber
\sum_{l = -\infty}^{\infty}|c_{j-l}||c_{k-l}| &\stackrel{n',n,\delta}{\ll}&
\left\{ \begin{array}{ll} 1 + \sum_{\substack{l=-\infty\\l\neq j}}^{\infty}|j-l|^{-2\sigma},\quad &j=k\\ \sum_{\substack{l=-\infty\\l\not\in \{j,k\}}}^{\infty} |j-l|^{-\sigma} |k-l|^{-\sigma}+|j-k|^{-\sigma},\quad &j\neq k\end{array}\right.\\
&\stackrel{n',n,\delta}{\ll}&
\left\{ \begin{array}{ll} 1,\quad &j=k\\ |j-k|^{1-2\sigma},\quad &j\neq k\end{array}\right.
\end{eqnarray}
Using (\ref{vassump},\ref{main2}-\ref{both1}), the summability of $|c_{l}^2|$ over integers, and recalling that $V_3=V_4$, we get that\\
\begin{eqnarray}\label{main3}\nonumber
\!\!\!\!\!&&\!\!\!\!
S(V_1,\hdots,V_6,\nu)
\\ \nonumber
\!\!\!\!\!&\stackrel{n',n,\delta}{\ll} &\!\!\!\!
\rho_{\left|V_2\right|,\hdots,\left|V_6\right|}\sum\limits_{k=n'+1}^{n}\!\!\left(\!1\!+\!\!\sum_{\substack{j=n'+1\\j\neq k}}^{n}\!|j-k|^{-2\sigma \left|V_5\right|} |j-k|^{-(\left|V_3\right|+\left|V_4\right|)\sigma} |j-k|^{(1-2\sigma)\left|V_2\right|}\!\right)
\\
\!\!\!\!\!&\stackrel{n',n,\delta}{\ll} &\!\!\!\!
\rho_{\left|V_2\right|,\hdots,\left|V_6\right|}\sum\limits_{k=n'+1}^{n}\!\!\left(\!1+\!\!\sum_{\substack{j=n'+1\\j\neq k}}^{n}\! |j-k|^{\left|V_2\right|-2(\left|V_2\right|+\left|V_3\right|+\left|V_5\right|)\sigma}\right).
\end{eqnarray}
(\ref{main3}) provides a bound for $\ S(V_1,\hdots,V_6,\nu)\ $ in terms of the cardinalities $\ \left|V_2\right|,\hdots\left|V_6\right|$. However, depending on the given partition $\lambda_q=(a_1,a_2,\hdots,a_q)$, the value of $v$ can be different, thus putting constraints on $V_2,\hdots,V_6$. We shall use (\ref{YS}) and (\ref{main3}) to bound the second moment of $Y_{n',n,\delta}^{\lambda_q}$.\\\\
\textbf{Case 1:} $a_q\geq 2$.\\
In this case, we see that $a_r\geq 2,\ \forall\ 1\leq r\leq q$, i.e. none of the $\psi$'s are zero-mean. Thus, Definition 4 gives us that $\left|V_2\right|=\left|V_3\right|=0$. Also from (\ref{vassump}), $\left|V_3\right|=\left|V_4\right|$ gives us that $\left|V_4\right|=0$. If further, $\left|V_5\right|=0$, then we will have $\left|V_6\right|=q$ (since $\left|V_2\right|+\hdots+\left|V_6\right|=q$). So by Lemma 1, we see that
$$\left|E\left(\prod_{r=1}^{q}(\psi_{l_r}^{(r)}\psi_{m_r}^{(r)})\right) - E\left(\prod_{r=1}^{q}\psi_{l_r}^{(r)}\right)E\left(\prod_{r=1}^{q}\psi_{m_r}^{(r)}\right)\right|=0$$
and hence $S(V_1,\hdots,V_6,\nu)=0$. Since we need not consider cases where\\$S(V_1,\hdots,V_6,\nu)=0$, we assume that $\left|V_5\right|\geq 1$. Thus, we have $\left|V_2\right|=\left|V_3\right|=0,\ \left|V_5\right|\geq 1$, and get that $\rho_{0,0,0,\left|V_5\right|,\left|V_6\right|}=\delta$ (from (\ref{rhobound})), and that $\left|V_2\right|-2(\left|V_2\right|+\left|V_3\right|+\left|V_5\right|)\sigma < -1$ (since $\sigma\in\left(\frac12,1\right)$). From (\ref{YS},\ref{main3}), we get that
\begin{eqnarray} \nonumber\label{case1}
E\left[(Y_{n',n,\delta}^{\lambda_q})^2\right]&\!\!\!\stackrel{n',n,\delta}{\ll}&\!\!\!\max_{\left|V_5\right|\geq 1,\ \left|V_6\right|}\rho_{0,0,0,\left|V_5\right|,\left|V_6\right|}\sum\limits_{k=n'+1}^{n}\left(1\ +\sum_{\substack{j=n'+1\\j\neq k}}^{n} |j-k|^{-2\left|V_5\right|\sigma}\right)
\\
&=&\!\!\!\delta\ (n-n')\ .
\end{eqnarray}
\textbf{Case 2:} $a_1=1$.\\
In this case, we see that $a_r= 1,\ \forall\ 1\leq r\leq q$, i.e. all the $\psi$'s are zero-mean. Thus, Definition 4 gives us that $\left|V_4\right|=\left|V_5\right|=\left|V_6\right|=0$. Also from (\ref{vassump}), $\left|V_3\right|=\left|V_4\right|$ gives us that $\left|V_3\right|=0$ and $\left|V_2\right|=q$. Since $a_1+a_2+\hdots+a_q=s$, and $a_r=1$ for each $r$, we have $q=s$ and hence, $\left|V_2\right|=s$. Thus, we have $\left|V_3\right|=\left|V_5\right|=0,\ \left|V_2\right|=s$, and get that $\rho_{s,0,0,0,0}=\delta$ (from (\ref{rhobound})), and that $\left|V_2\right|-2(\left|V_2\right|+\left|V_3\right|+\left|V_5\right|)\sigma = (1-2\sigma)s$. From (\ref{YS},\ref{main3}), we get that
\begin{eqnarray}\label{case2}\nonumber
E\left[(Y_{n',n,\delta}^{\lambda_q})^2\right] &\stackrel{n',n,\delta}{\ll}& \rho_{s,0,0,0,0}\sum\limits _{k=n'+1}^{n}\left(1\ \ +\ \sum_{\substack{j=n'+1\\j\neq k}}^{n} |j-k|^{(1-2\sigma)s}\right)
\\
&\stackrel{n',n,\delta}{\ll}&
\delta\ (n-n')\ \ l_{s,\sigma}(n-n')\ ,
\end{eqnarray}
where $l_{s,\sigma}$ is from the Notation list in Subsection 1.2.\\\\
\textbf{Case 3:} $a_1\geq 2,\ a_q=1$.\\
In this case, we have at least one, but not all zero-mean $\psi$'s. Notice from (\ref{vassump}) and Definition 3, that $0<\left|V_2\right|+\left|V_3\right|<q$ and $0<\left|V_4\right|+\left|V_5\right|+\left|V_6\right|<q$. First, assume that $\left|V_3\right|=\left|V_5\right|=0$. Since from (\ref{vassump}), we have $\left|V_3\right|=\left|V_4\right|$, thus we get that $\left|V_4\right|=0$, and $\ \left|V_2\right|,\left|V_6\right|\in\{1,2,\hdots,q-1\}$. So we have $\rho_{\left|V_2\right|,0,0,0,\left|V_6\right|}=1$ (from (\ref{rhobound})), and that $\left|V_2\right|-2(\left|V_2\right|+\left|V_3\right|+\left|V_5\right|)\sigma = (1-2\sigma)\left|V_2\right|$, hence using (\ref{YS},\ref{main3}), we get that
\small{\begin{eqnarray}\label{case31}\nonumber
	E\left[(Y_{n',n,\delta}^{\lambda_q})^2\right]\!\!
	&\stackrel{n',n,\delta}{\ll}&\!\! \max_{\left|V_2\right|,\left|V_6\right|\in\{1,2,\hdots,q-1\}}\rho_{\left|V_2\right|,0,0,0,\left|V_6\right|}\! \sum\limits_{k=n'+1}^{n}\!\!\left(1 +\!\sum_{\substack{j=n'+1\\j\neq k}}^{n} |j-k|^{(1-2\sigma)\left|V_2\right|}\!\right)
	\\
	&\stackrel{n',n,\delta}{\ll}&\!\!
	\max_{\left|V_2\right|\in\{1,2,\hdots,q-1\}}\ \ (n-n')\ \ l_{\left|V_2\right|,\sigma}(n-n')\ .
	\end{eqnarray}}
\normalsize \!\!Note that $ l_{\left|V_2\right|,\sigma}(n-n')=\left\{\begin{array}{ll}
(n-n')^{\left|V_2\right|(1-2\sigma)+1}\!,\ &\sigma <\frac{\left|V_2\right|+1}{2\left|V_2\right|}\\
\log(n-n'+1),\ &\sigma =\frac{\left|V_2\right|+1}{2\left|V_2\right|}\\
1,\ &\sigma >\frac{\left|V_2\right|+1}{2\left|V_2\right|}
\end{array}\right.$, from Notation List in Subsection 1.2. Since $(1-2\sigma)<0$, $\ n-o\geq 1$, and $\frac{\left|V_2\right|+1}{2\left|V_2\right|}$ decreases as $\left|V_2\right|$ increases, observe that $l_{\left|V_2\right|,\sigma}(n-n')$ is a non-increasing function of $\left|V_2\right|\in\{1,2,\hdots,v\}$. Thus, we take $\left|V_2\right|=1$ in (\ref{case31}) to bound the left hand side, and get
\begin{eqnarray}\label{case31.5}
E\left[(Y_{n',n,\delta}^{\lambda_q})^2\right]
&\stackrel{n',n,\delta}{\ll}&
(n-n')\ \ l_{1,\sigma}(n-n')\ .
\end{eqnarray}
For all other values of $\left|V_3\right|$ and $\left|V_5\right|$, we have $\left|V_2\right|-2(\left|V_2\right|+\left|V_3\right|+\left|V_5\right|)\sigma<-1$ (since $\sigma\in\left(\frac12,1\right)$), and $\rho_{\left|V_2\right|,\hdots,\left|V_6\right|}\leq \delta$ (from (\ref{rhobound})). Thus, we get from (\ref{YS},\ref{main3}), that
\small{\begin{eqnarray}\label{case32}\nonumber
	E\left[(Y_{n',n,\delta}^{\lambda_q})^2\right]\!\!\!\!
	&\stackrel{n',n,\delta}{\ll}&\!\!\!\!\!
	\max_{\left|V_2\right|,\hdots,\left|V_6\right|}\  \rho_{\left|V_2\right|,\hdots,\left|V_6\right|}\!\sum_{k=n'+1}^{n}\!\!\left(\!1\! +\!\!\!\sum_{\substack{j=n'+1\\j\neq k}}^{n}\! |j-k|^{\left|V_2\right|-2(\left|V_2\right|+\left|V_3\right|+\left|V_5\right|)\sigma}\!\right)
	\\
	&\stackrel{n',n,\delta}{\ll}&\!\!\!\!\!
	\ \delta\ (n-n')\ .
	\end{eqnarray}}\normalsize
\textbf{Extra Case:} $a_1\geq 2,\ a_q=1$, $s$ is even, and $E\left(\psi_1^{(r)}\right) = 0\ $whenever $a_r$ is odd.\\
Under these new conditions, we will show that it is possible to tighten the bound for $E\left[(Y_{n',n,\delta}^{\lambda_q})^2\right]$ in (\ref{case31}). We had previously taken $\left|V_2\right|=1$ to bound $E\left[(Y_{n',n,\delta}^{\lambda_q})^2\right]$ in (\ref{case31}) of Case 3, under the assumption that $\left|V_3\right|=\left|V_4\right|=\left|V_5\right|=0\ $ and $\left|V_2\right|,\left|V_6\right|\in\{1,2,\hdots,q-1\}$. Further, when $\left|V_2\right|=1$, it means that $\psi_{l_q}^{(q)}$ and $\psi_{m_q}^{(q)}$ are the only two $\psi$'s with zero mean, and that they must be matched. This gives us that $\nu(q)=q,\ \left|V_6\right|=q-1\ $ and that $V_1\cup V_2\cup V_3\neq\phi$. So, we apply (3.6) from the paper, and the independence of $\psi$'s with different subscripts, to the definition of $S(V_1,\hdots,V_6,\nu)$ in (\ref{Sdef}), and get that
\begin{eqnarray}\nonumber\label{cov}
&&
\left|E\left(\prod_{r=1}^{q}\psi_{l_r}^{(r)}\psi_{m_r}^{(r)}\!\right) - E\left(\prod_{r=1}^{q}\psi_{l_r}^{(r)}\!\right)E\left(\prod_{r=1}^{q}\psi_{m_r}^{(r)}\!\right)\right|\textbf{1}_{\{\left|V_2\right|=1,\ \left|V_6\right|=q-1\}}
\\ \nonumber
&= &
\left|E\left(\prod_{r=1}^{q}\psi_{l_r}^{(r)}\psi_{m_r}^{(r)}\right)\right|\textbf{1}_{\{\left|V_2\right|=1,\ \left|V_6\right|=q-1\}}
\\
&= &
\left|\prod_{r=1}^{q-1}E\left(\psi_{l_r}^{(r)}\right)\right|\left|\prod_{r=1}^{q-1}E\left(\psi_{m_r}^{(r)}\right)\right|E\left[\left(\psi_{l_q}^{(q)}\right)^2\right]\ .
\end{eqnarray}
Observe that $(a_1,a_2,\hdots,a_{q-1})$ is a decreasing partition of $(s-1)$, since $a_q = 1$. Hence if $s$ is even, then $a_r$ must be odd for some $1\leq r\leq q-1$, and for that $r$, we will get $E\left(\psi_{l_r}^{(r)}\right) = 0$. This makes the entire expression in (\ref{cov}) become $0$ (thus making $S(V_1,\hdots,V_6,\nu)=0$), so we must not choose $\left|V_2\right|=1\ $ for the bound of $E\left[\left(Y_{n',n,\delta}^{\lambda_q}\right)^2\right]\ $ in (\ref{case31}). Since $l_{\left|V_2\right|,\sigma}(n-n')$ is a non-increasing function of $\left|V_2\right|$, we go with next lowest value, i.e. $\left|V_2\right|=2$, to obtain
\begin{eqnarray}\label{extracase}
E\left[\left(Y_{n',n,\delta}^{\lambda_q}\right)^2\right]\ \ \stackrel{n',n,\delta}{\ll}\ \ (n-n')\ \ l_{2,\sigma}(n-n')\ .
\end{eqnarray}\\
\noindent
Lemma \ref{psilemma} follows from (\ref{case1},\ref{case2},\ref{case31.5},\ref{case32}) and (\ref{extracase}).

\end{proof}

\begin{manuallemma}{3}\label{inequality}
	For $j, k\in\mathbb Z$, $j \neq k$ and $\gamma > \frac12$, we have,
	$$\sum_{\substack{l=-\infty\\l\not\in \{j,k\}}}^{\infty} |j-l|^{-\gamma} |k-l|^{-\gamma} \quad \stackrel{j,k}{\ll} \quad
	\left\{ \begin{array}{ll} |j-k|^{1-2\gamma}, & \gamma\in \left(\frac12,1\right)\\ |j-k|^{-1} \ln(|j-k|+1),\quad & \gamma = 1\\|j-k|^{-\gamma},& \gamma >1\end{array}\right..$$
\end{manuallemma}
\begin{proof}
Take $\gamma\in \left(\frac12,1\right)$.
Without loss of generality, we take $j > k$. 
Using symmetry, integral approximation, and successive substitutions $t=|k-l|\ $ and $\ s=\frac{t}{j-k}$, we get
	\begin{eqnarray}\nonumber\label{gammaless1}
	\!\!\!\!&&\sum_{\substack{l=-\infty\\l\not\in \{j,k\}}}^{\infty} |j-l|^{-\gamma} |k-l|^{-\gamma}
	\\ \nonumber
	\!\!\!\!&\stackrel{j,k}{\ll}&
	\sum_{l=-\infty}^{k-1} (j-l)^{-\gamma}(k-l)^{-\gamma}\ \ +\ \sum_{l=k+1}^{j-1}(j-l)^{-\gamma}(l-k)^{-\gamma}
	\\ \nonumber
	\!\!\!\!&\stackrel{j,k}{\ll}&
	\int_{0}^{\infty}(j-k+t)^{-\gamma}t^{-\gamma}\ dt\ \ +\ \int_{0}^{j-k}\!(j-k-t)^{-\gamma}t^{-\gamma}\ dt
	\\
	\!\!\!\!&\stackrel{j,k}{\ll}&
	(j-k)^{1-2\gamma} \left(\int_{0}^{\infty} (1+s)^{-\gamma}s^{-\gamma}\ ds\ \ +\int_{0}^{1} (1-s)^{-\gamma}s^{-\gamma}\ ds\right).
	\end{eqnarray}
Notice that $\ \int_{0}^{1} (1+s)^{-\gamma}s^{-\gamma}\ ds\ \leq\ \int_{0}^{1} (1-s)^{-\gamma}s^{-\gamma}\ ds = B(1-\gamma,1-\gamma)$, which is the beta function evaluated at $(1-\gamma,1-\gamma)$. 
Hence, we get by (\ref{gammaless1}) that
	\begin{eqnarray}\nonumber\label{gammales1}
	\!\!\!\!&&\sum_{\substack{l=-\infty\\l\not\in \{j,k\}}}^{\infty} |j-l|^{-\gamma} |k-l|^{-\gamma}
	\\ \nonumber
	\!\!\!\!&\stackrel{j,k}{\ll}&
	(j-k)^{1-2\gamma} \left(\int_{1}^{\infty} (1+s)^{-\gamma}s^{-\gamma}\ ds\ +\ 2\int_{0}^{1} (1-s)^{-\gamma}s^{-\gamma}\ ds\right)
	\\
	\!\!\!\!&\stackrel{j,k}{\ll}&
	(j-k)^{1-2\gamma}.
	\end{eqnarray}
	Next, we consider the case where $\gamma = 1$.
	\begin{eqnarray}\nonumber\label{gammais1}
	\!&&\sum_{\substack{l=-\infty\\l\not\in \{j,k\}}}^{\infty} |j-l|^{-1} |k-l|^{-1}
	\\ \nonumber
	\!&\stackrel{j,k}{\ll} &
	\sum_{l=-\infty}^{k-1} (j-l)^{-1} (k-l)^{-1}\ \ +\ \sum_{l=k+1}^{j-1} (j-l)^{-1} (l-k)^{-1}
	\\ \nonumber
	\!&= &
	(j-k)^{-1}\!\left(\sum_{l=-\infty}^{k-1} \left[(k-l)^{-1}-(j-l)^{-1}\right]+\sum_{l=k+1}^{j-1} \left[(l-k)^{-1} + (j-l)^{-1}\right]\right)
	\\
	\!&= &
	(j-k)^{-1}\left(\sum_{l=1}^{j-k} l^{-1}\ +\ 2\!\sum_{l=1}^{j-k-1}l^{-1}\!\right)
	\ \ \stackrel{j,k}{\ll}\ \ (j-k)^{-1} \log(j-k+1).
	\end{eqnarray}
	Finally we consider the case where $\gamma > 1$. Using symmetry, and summability of the sequence $\{|l|^{-\gamma}\}_{l\in\mathbb{Z}}$, we have
	\begin{eqnarray}\label{gammabig1}\nonumber
	\sum_{\substack{l=-\infty\\l\not\in \{j,k\}}}^{\infty} |j-l|^{-\gamma} |k-l|^{-\gamma}\ 
	&\stackrel{j,k}{\ll}&
	\sum_{\substack{l=-\infty\\l\neq k}}^{\left\lfloor \frac{j+k}{2}\right\rfloor }\left(j-l\right)^{-\gamma}\left|k-l\right|^{-\gamma}
	\\ \nonumber
	&\stackrel{j,k}{\ll}&
	\left(j-\left\lfloor \frac{j+k}{2}\right\rfloor\right)^{-\gamma}\ \sum_{\substack{l=-\infty\\l\neq k}}^{\left\lfloor \frac{j+k}{2}\right\rfloor}\left|k-l\right|^{-\gamma}
	\\
	&\stackrel{j,k}{\ll}&
	(j-k)^{-\gamma}
	\end{eqnarray}
	From (\ref{gammales1}, \ref{gammais1}) and (\ref{gammabig1}), the proof of the lemma is complete.\hfill
\end{proof}

\begin{manuallemma}{4}\label{newineq}
	For $j, k\in\mathbb Z$, $j \neq k$ and $\gamma \in \left(\frac12,1\right)$, we have,
	$$\sum_{\substack{l=-\infty\\l\not\in \{j,k\}}}^{\infty} |j-l|^{-\gamma} |k-l|^{-2\gamma} \qquad \stackrel{j,k}{\ll} \qquad |j-k|^{-\gamma}\ .$$
\end{manuallemma}
\begin{proof}
	Without loss of generality, we assume that $j>k$. Then, we have
	\begin{eqnarray}\nonumber
	&&\sum_{\substack{l=-\infty\\l\not\in \{j,k\}}}^{\infty} |j-l|^{-\gamma} |k-l|^{-2\gamma}
	\\ \nonumber
	&\stackrel{j,k}{\ll} &
	\sum_{\substack{l=-\infty\\l\neq k}}^{\left\lfloor \frac{j+k}{2}\right\rfloor }|j-l|^{-\gamma} |k-l|^{-2\gamma}\ \ +\ \sum_{\substack{l=\left\lceil \frac{j+k}{2}\right\rceil\\l\neq j}}^{\infty} |j-l|^{-\gamma} |k-l|^{-2\gamma}
	\\ \nonumber
	&\stackrel{j,k}{\ll}&
	|j-k|^{-\gamma}\sum_{\substack{l=-\infty\\l\not\in \{j,k\}}}^{\infty}|k-l|^{-2\gamma}\ +\  |j-k|^{-\gamma}\sum_{\substack{l=-\infty\\l\not\in \{j,k\}}}^{\infty}|j-l|^{-\gamma} |k-l|^{-\gamma}
	\\
	&\stackrel{j,k}{\ll}&
	|j-k|^{-\gamma},
	\end{eqnarray}
	by Lemma \ref{inequality}. This concludes the proof of the lemma.\hfill
\end{proof}

\begin{manualtheorem}{5}\label{stout}
	Let $\{Z_k\}_{k\in\mathbb{N}}$ be a time series with finite second moments, and $f$ be a super-additive function on $\mathbb{N}$, such that
	\begin{eqnarray}\label{serflingcond}
	E\left[\left(\sum_{i=n'+1}^{n} Z_i \right)^2\right]\ &\leq&\ f(n-{n'})\qquad\!\forall\ n'<n\in\mathbb{N}\cup\{0\}\ .
	\end{eqnarray}
	Then, for $n_{r}=2^{r},\ r\in\mathbb{N}\cup\{0\}$, and $n',n\in\mathbb{N}$, we have
	\begin{eqnarray}\nonumber
	E\left[\max_{n_r\leq n'<n<n_{r+1}}\left(\sum_{i=n'+1}^{n} Z_i\right)^2 \right]\ \ \stackrel{r}\ll\ \ r^2 f(n_r)\ .
	\end{eqnarray}
\end{manualtheorem}

\begin{proof}
    Notice that
    \begin{align}\nonumber\label{ultimatestep}
        E\!\left[\max_{n_r\leq n'<n< n_{r+1}} \left(\sum_{i=n'+1}^{n}Z_i\right)^2\right] &= E\!\left[\max_{n_r\leq n'<n< n_{r+1}}\! \left(\sum_{i=n_r+1}^{n}Z_i -\! \sum_{i=n_r+1}^{n'}Z_i\right)^2\right]\\
        &\leq\ 4 E\!\left[\max_{n_r\leq n< n_{r+1}}\! \left(\sum_{i=n_r+1}^{n}Z_i\right)^2\right]\ .
    \end{align}
    Thus, Theorem \ref{stout} will be proved if we can show that the right hand side of (\ref{ultimatestep}) is upper bounded up to a constant by $r^2 f(n_r)$.
    In the setting of Stout [33, Theorem 2.4.1], taking $\nu=2$, $\{X_i\}_{i\in\mathbb{N}} = \{Z_i\}_{i\in\mathbb{N}}$, and $g(F_{a,k})=f(k)$, which is a super-additive function, we see that $g(F_{a,k})+g(F_{a+k,m})\leq g(F_{a,k+m})$, i.e. the condition [33, (2.4.1)] is satisfied. Thus, for $n_{r}=2^{r}$ where $r\in\mathbb{N}\cup\{0\}$, we have
    \begin{align}\nonumber\label{M}
        E\left[\max_{n_r\leq n< n_{r+1}} \left(\sum_{i=n_r+1}^{n}Z_i\right)^2\right] &\leq  E\left[\max_{n_r<k\leq n_{r+1}} \left(\sum_{i=n_r+1}^{k}Z_i\right)^2\right]\\\nonumber
        &\leq\ \left(\frac{\log(2 n_r)}{\log(2)}\right)^2 g(F_{n_r,n_r})\\
        &\stackrel{r}{\ll}\ r^2 f(n_r)\ .
    \end{align}
    Theorem \ref{stout} follows from (\ref{M}) and (\ref{ultimatestep}).
\end{proof}

\bibliographystyle{unsrt}      
\bibliography{MSLLNprod_Arxiv}   

\begin{thebibliography}{10}

\bibitem{relevance}
M.~Grossglauser and J.~C. Bolot.
\newblock On the relevance of long-range dependence in network traffic.
\newblock {\em IEEE/ACM Trans. Netw.}, 7:629--640, 1999.

\bibitem{runoff}
E.~Koscielny‐Bunde, J.~W. Kantelhardt, P.~Braun, A.~Bunde, and S.~Havlin.
\newblock Long‐term persistence and multifractality of river runoff records:
  Detrended fluctuation studies.
\newblock {\em J. Hydrol.}, 322:120--137, 2006.

\bibitem{patagonia}
M.M. Rodrigo and A.Q. Renato.
\newblock Long-range dependence in the runoff time series of the most important
  patagonian river draining to the pacific ocean.
\newblock {\em New Zeal. J. Mar. Fresh. Res.}, 52:264--283, 2017.

\bibitem{stockcrash}
D.~Sornette.
\newblock {\em Why Stock Markets Crash: Critical Events in Complex Financial
  Systems}.
\newblock Princeton University Press, 2009.

\bibitem{KOSA}
M.A. Kouritzin and S.~Sadeghi.
\newblock Marcinkiewicz law of large numbers for outer products of
  heavy-tailed, long-range dependent data.
\newblock {\em Adv. Appl. Probab.}, 48:349--368, 2016.

\bibitem{louhichi}
S.~Louhichi and P.~Soulier.
\newblock Marcinkiewicz–zygmund strong laws for infinite variance time
  series.
\newblock {\em Statistical Inference for Stochastic Processes}, 3:31--40, 2000.

\bibitem{Rio}
E.~Rio.
\newblock A maximal inequality and dependent marcinkiewicz-zygmund strong laws.
\newblock {\em Ann. Probab.}, 23:918--937, 1995.

\bibitem{chistyakov}
V.~P. Chistyakov.
\newblock A theorem on sums of independent positive random variables and its
  applications to branching random processes.
\newblock {\em Theory Probab. its Appl.}, 9:640--648, 1964.

\bibitem{Teugels}
J.~L. Teugels.
\newblock The class of subexponential distributions.
\newblock {\em Ann. Probab.}, 3:1000--1011, 1975.

\bibitem{Mand60}
B.~Mandelbrot.
\newblock The pareto-lévy law and the distribution of income.
\newblock {\em Int. Econ. Rev.}, 1:79--106, 1960.

\bibitem{Mand61}
B.~Mandelbrot.
\newblock Stable paretian random functions and the multiplicative variation of
  income.
\newblock {\em Econometrica}, 29:517--543, 1961.

\bibitem{kuliksoulier}
R.~Kulik and P.~Soulier.
\newblock {\em Heavy-Tailed Time Series}.
\newblock Springer-Verlag New York, 2020.

\bibitem{Hurst}
H.E. Hurst.
\newblock The problem of long-term storage in reservoirs.
\newblock {\em Hydrol. Sci. J.}, 1:13--27, 1956.

\bibitem{otherhurst}
H.E. Hurst.
\newblock Methods of using long-term storage in reservoirs.
\newblock {\em Proceedings of the Institution of Civil Engineers}, 5:519--543,
  1956.

\bibitem{noahjoseph}
B.~Mandelbrot and J.~R. Wallis.
\newblock Noah, joseph and operational hydrology.
\newblock {\em Water Resour. Res.}, 4:909--918, 1968.

\bibitem{internet}
T.~Karagiannis, M.~Molle, and M.~Faloutsos.
\newblock Long-range dependence ten years of internet traffic modeling.
\newblock {\em IEEE Internet Comput.}, 8:57--64, 2004.

\bibitem{mingliu}
M.~Liu.
\newblock Modeling long memory in stock market volatility.
\newblock {\em J. Econom.}, 99:139--171, 2000.

\bibitem{history}
G.~Graves, R.~Gramacy, N.~Watkins, and C.~Franzke.
\newblock A brief history of long memory: Hurst, mandelbrot and the road to
  arfima, 1951–1980.
\newblock {\em Entropy}, 19:437, 2017.

\bibitem{Hosking}
J.~R.~M. Hosking.
\newblock Fractional differencing.
\newblock {\em Biometrika}, 68:165--176, 1981.

\bibitem{taqqubook}
V.~Pipiras and M.~Taqqu.
\newblock {\em Long-Range Dependence and Self-Similarity}.
\newblock Cambridge University Press, 2017.

\bibitem{Davisresnick}
R.~Davis and S.~Resnick.
\newblock Limit theory for moving averages of random variables with regularly
  varying tail probabilities.
\newblock {\em Ann. Probab.}, 13:179--195, 1985.

\bibitem{crosscov}
M.A. Kouritzin.
\newblock Strong approximation for cross-covariances of linear variables with
  long-range dependence.
\newblock {\em Stoch. Process. Their Appl.}, 60:343--353, 1995.

\bibitem{Wudepinnov}
W.~B. Wu and W.~Min.
\newblock On linear processes with dependent innovations.
\newblock {\em Stoch. Process. Their Appl.}, 115:939--958, 2005.

\bibitem{covlongmem}
W.~B. Wu, Y.~Huang, and W.~Zheng.
\newblock Covariances estimation for long-memory processes.
\newblock {\em Adv. Appl. Probab.}, 42:137--157, 2010.

\bibitem{dangistas}
T.~T.~N. Dang and J.~Istas.
\newblock Estimation of the hurst and the stability indices of a h-self-similar
  stable process.
\newblock {\em Electron. J. Stat.}, 11:4103–4150, 2017.

\bibitem{kosadecoup}
M.A. Kouritzin and S.~Sadeghi.
\newblock Convergence rates and decoupling in linear stochastic approximation
  algorithms.
\newblock {\em SIAM J. Control Optim.}, 53:1484--1508, 2015.

\bibitem{interrelation}
M.A. Kouritzin.
\newblock On the interrelation of almost sure invariance principles for certain
  stochastic adaptive algorithms and for partial sums of random variables.
\newblock {\em J. Theor. Probab.}, 9:811--840, 1996.

\bibitem{Shiryaev}
A.~N. Shiryaev.
\newblock {\em Probability}.
\newblock Springer, 1996.

\bibitem{Samorod}
G.~Samorodnitsky.
\newblock {\em Stochastic Processes and Long Range Dependence}.
\newblock Springer, 2016.

\bibitem{baitaqqu}
S.~Bai and M.~S. Taqqu.
\newblock Convergence of long-memory discrete kth order volterra processes.
\newblock {\em Stoch. Process. Their Appl.}, 125:2026--2053, 2015.

\bibitem{peccatitaqqu}
G.~Peccati and M.~S. Taqqu.
\newblock {\em Combinatorial Expressions of Cumulants and Moments}.
\newblock Springer, Milan, 2011.

\bibitem{Surgzones}
D.~Surgailis.
\newblock Zones of attraction of self-similar multiple integrals.
\newblock {\em Lith. Math. J.}, 22:327–340, 1982.

\bibitem{Surglrdappell}
D.~Surgailis.
\newblock Long-range dependence and appell rank.
\newblock {\em Ann. Probab.}, 28:478--497, 2000.

\bibitem{Surgemp}
D.~Surgailis.
\newblock Stable limits of empirical processes of moving averages with infnite
  variance.
\newblock {\em Stoch. Process. Their Appl.}, 100:255–274, 2002.

\bibitem{Surgstable}
D.~Surgailis.
\newblock Stable limits of sums of bounded functions of long-memory moving
  averages with finite variance.
\newblock {\em Bernoulli}, 10:327--355, 2004.

\bibitem{breurmajor}
P.~Breuer and P.~Major.
\newblock Central limit theorems for non-linear functionals of gaussian random
  fields.
\newblock {\em J. Multivariate Anal.}, 13:425--441, 1983.

\bibitem{dobrushinmajor}
R.~L. Dobrushin and P.~Major.
\newblock Non-central limit theorems for non-linear functions of gaussian
  fields.
\newblock {\em Z. Wahrscheinlichkeitstheorie Verw. Geb.}, 50:27--52, 1979.

\bibitem{taqqu}
M.~S. Taqqu.
\newblock Convergence of integrated processes of arbitrary hermite rank.
\newblock {\em Z. Wahrscheinlichkeitstheorie Verw. Geb.}, 50:53–83, 1979.

\bibitem{avramtaqqu}
F.~Avram and M.~S. Taqqu.
\newblock Noncentral limit theorems and appell polynomials.
\newblock {\em Ann. Probab.}, 15:767--775, 1987.

\bibitem{vaiciulis}
M.~Vaičiulis.
\newblock Convergence of sums of appell polynomials with infinite variance.
\newblock {\em Lithuanian Math. J.}, 43:80--98, 2003.

\bibitem{Loeve}
M.~Loeve.
\newblock {\em Probability Theory I}.
\newblock Springer-Verlag New York, 1977.

\bibitem{Stout}
W.~F. Stout.
\newblock {\em Almost Sure Convergence}.
\newblock Academic Press, 1974.

\end{thebibliography}

\end{document}